\documentclass[pdflatex,sn-aps]{sn-jnl}

\usepackage{graphicx}%
\usepackage{multirow}%
\usepackage{amsmath,amssymb,amsfonts}%
\usepackage{amsthm}%
\usepackage{mathrsfs}%
\usepackage[title]{appendix}%
\usepackage{xcolor}%
\usepackage{textcomp}%
\usepackage{manyfoot}%
\usepackage{booktabs}%
\usepackage{algorithm}%
\usepackage{algorithmic}%
\usepackage{listings}%

\usepackage[english]{babel}
\usepackage{float}
\usepackage{graphicx}
\usepackage{subcaption}
\usepackage{enumerate}
\usepackage{booktabs}
\usepackage{dsfont}
\usepackage{comment}
\usepackage{bm}
\usepackage{microtype}

\usepackage{todonotes}
\usepackage{tablefootnote}

\usepackage{tikz}
\usepackage{pgfplots}
\pgfplotsset{compat=1.17}
\pgfplotsset{minor tick style={draw=none}}
\usepackage{caption}
\usepackage{float}

\allowdisplaybreaks[4]

\numberwithin{equation}{section}
\theoremstyle{thmstyleone}%
\newtheorem{theorem}{Theorem}[section]

\newtheorem{lemma}{Lemma}[section]
\newtheorem{assumption}{Assumption}[section]
\newtheorem{proposition}{Proposition}[section]

\theoremstyle{thmstyletwo}%

\newtheorem{remark}{Remark}[section]

\theoremstyle{thmstylethree}%
\newtheorem{definition}{Definition}%

\def\cL{{\mathcal{L}}}
\def\R{{\mathbb{R}}}

\def\Argmin{\mathop{\rm Arg\,min}}

\definecolor{cb-black}      {RGB}{  0,   0,   0}
\definecolor{cb-blue-green} {RGB}{  0,  073,  073}
\definecolor{cb-green-sea}  {RGB}{  0, 146, 146}
\definecolor{cb-rose}       {RGB}{255, 109, 182}
\definecolor{cb-salmon-pink}{RGB}{255, 182, 119}
\definecolor{cb-purple}     {RGB}{ 73,   0, 146}
\definecolor{cb-blue}       {RGB}{ 0, 109, 219}
\definecolor{cb-lilac}      {RGB}{182, 109, 255}
\definecolor{cb-blue-sky}   {RGB}{109, 182, 255}
\definecolor{cb-blue-light} {RGB}{182, 219, 255}
\definecolor{cb-burgundy}   {RGB}{146,   0,   0}
\definecolor{cb-brown}      {RGB}{146,  73,   0}
\definecolor{cb-clay}       {RGB}{219, 209,   0}
\definecolor{cb-green-lime} {RGB}{ 36, 255,  36}
\definecolor{cb-yellow}     {RGB}{255, 255, 109}

\begin{document}

\title[A single-loop CG-AL framework]{A conditional-gradient-based single-loop augmented Lagrangian method for inequality constrained optimization}


\author[1]{Xiaozhou Wang}\email{xiaozhou.wang@m.scnu.edu.cn}

\author[2]{Ting Kei Pong}\email{tk.pong@polyu.edu.hk}

\author*[3]{Zev Woodstock}\email{woodstzc@jmu.edu}

\affil[1]{School of Mathematical Sciences, South China Normal University, Guangzhou, People's Republic of China}

\affil[2]{Department  of  Applied  Mathematics, The  Hong  Kong  Polytechnic  University, Hong Kong, People's Republic of China}

\affil[3]{Department of Mathematics and Statistics, James Madison University, Harrisonburg, Virginia, USA}

\abstract{We consider the problem of minimizing the sum of a Lipschitz differentiable convex function $f$ and a proper closed convex function $h$ that admits efficient linear minimization oracles, subject to multiple smooth convex inequality constraints. We adapt the classical augmented Lagrangian (AL) method for these problems: in each iteration, our algorithm consists of {\em one step} of the conditional gradient (CG) method applied to the AL function, followed by an update of the dual variable as in classical AL methods with a {\em diminishing} dual stepsize. We study the convergence rate of our algorithm under two standard stepsize rules for the CG method, namely, an open-loop stepsize and the short stepsize, and obtain a convergence rate that matches the best-known complexity for this class of problems. We also establish accelerated rates when $h$ is the indicator function of a uniformly convex set.}




\maketitle

\section{Introduction}

Augmented Lagrangian (AL) methods are a classical family of optimization algorithms dating back over a half century \cite{Hest69,Powe69,Rockafellar-1973,Bert76,Kort76}. Over these decades, the connections between AL algorithms and proximal methods have been identified \cite{Rock76} and used to devise approaches to solve optimization problems involving inequality constraints; see, e.g., \cite{Bert82,Iuse99,Marc23} and the references therein. The rich interplay between AL methods and Newton-type algorithms, including single-loop algorithms, also dates back to roughly the same time \cite{Tapi77}.

On the other hand, the interplay between AL algorithms and \emph{Conditional Gradient} (CG, also known as `Frank-Wolfe') methods, i.e., algorithms that rely on linear minimization subproblems instead of proximal minimization subproblems, is notably newer. To the authors' knowledge, the first single-loop AL-type algorithm to rely on linear minimization subproblems was proposed in 2018, with several theoretical developments and CG+AL variants following in recent years \cite{Gidel-2018,Yurt19,Silveti-Molinari-Fadili-2020} (see also \cite[Appendix~A]{Garb23}). A notable advantage of this CG+AL approach is that, particularly in high-dimensional settings, the linear minimization subproblems often require fewer computational resources when compared to traditional proximity operators \cite{Freu16,Comb21LMO,Wood26}. In fact, it is currently an open question as to whether or not linear minimization \emph{always} requires less computation than solving a proximal subproblem \cite{Wood26}. These methods also exhibit a low numerical cost-per-iteration because they are \emph{single-loop} methods, i.e., convergence is proven while only relying on one ``internal'' CG step within each iteration (as compared to double-loop methods that rely on multiple internal CG steps for each outer iteration, e.g., \cite{HeHa15,LiuL19,Kolm21,Mill21}). While the CG+AL approaches mentioned thus-far exemplify important progress for the family of CG+AL algorithms, it appears that there are currently no CG+AL methods that explicitly account for, or leverage the structure of, convex inequality constraints. 

In view of the historical inclusion of inequality constraints for AL methods \cite{Bert76,Kort76,Rock76}, and in the interest of devising a CG+AL method that can leverage the structure of inequality constraints, this work considers the following optimization problem:
\begin{equation}\label{problem}
\min_{{\bm x}\in \R^n} f({\bm x}) + h({\bm x}) \quad   \mbox{s.t.  ${\bm g}(\bm x):=\begin{bmatrix}g_1(\bm x)&\cdots&g_m(\bm x)\end{bmatrix}^T\leq {\bm 0}$},
\end{equation}
where $f:\R^n\rightarrow \R$ is convex and Lipschitz differentiable, $h: \R^n\rightarrow \R\cup \{\infty\}$ is proper, closed and convex,  and  for each $i\in [m]:=\{1,\ldots,m\}$, each function $g_i:\R^n\rightarrow \R$ is differentiable and convex.
It is also assumed that, for any $\bm c\in\R^n$, the problem $\min_{\bm x\in\R^n}\langle \bm c,\bm x\rangle+h(\bm x)$ has a minimizer and can be solved efficiently using a \emph{linear minimization oracle (LMO)}; oftentimes, the function $h$ is used to represent another constraint that is accessible via a linear minimization oracle \cite{Beck-Pauwels-Sabach-2015,Silveti-Molinari-Fadili-2020}; the precise assumptions on \eqref{problem} will be presented in Assumption~\ref{assu-function} below. 
Problem \eqref{problem} appears in a variety of settings in data science and signal processing, e.g., radiotherapy treatment planning, link prediction, covariance estimation, and graph denoising \cite{Rich12,Gidel-2018,LanR21,Woodstock-Pokutta-2025}.

In addition to the aforementioned CG+AL methods \cite{Gidel-2018,Yurt19,Silveti-Molinari-Fadili-2020}, several single-loop CG methods for solving instances of \eqref{problem} have also been proposed \cite{LanR21,Asga24,Gian25,Woodstock-Pokutta-2025}. Although \eqref{problem} is technically a special case of the problems considered in \cite{Gian25,Woodstock-Pokutta-2025}, neither works explicitly account for inequality constraints. On the other hand, while \cite{Asga24,LanR21} account for inequality constraints, their algorithms and analysis differ from AL approaches.
To the best of our knowledge, both works \cite{Asga24,LanR21} exhibit the fastest convergence rate for solving \eqref{problem} with a single-loop CG algorithm -- guaranteeing that after $k$ iterations, both primal suboptimality and feasibility violation are bounded by $\mathcal{O}(1/\sqrt{k})$. 
Moreover, when $h={\rm Ind}_{\mathcal{C}}$ is an indicator function with $\mathcal{C}$ being compact and uniformly convex and ${\bm g}\equiv {\bm 0}$, the problem \eqref{problem} reduces to an important setting where the vanilla CG method applies. In this case, it is well established that the vanilla CG method enjoys accelerated convergence rates under standard stepsize rules, such as the open-loop stepsize and the short stepsize rules; see, e.g., \cite{Garber-Hazan-2015,Kerdreux-dAspremont-Pokutta-2021,Grimmer-Liu-2026,Halbey-et-2026}. However, it remains unknown whether LMO-based methods can achieve similar acceleration when $h={\rm Ind}_{\mathcal{C}}$ for a compact and uniformly convex set and ${\bm g}$ is nonzero.
This begs the following questions:

\begin{enumerate}[\rm (i)]
 \item Can a CG+AL method  explicitly leverage the structure of convex inequality constraints and match the current state-of-the-art convergence rate $\mathcal{O}(1/\sqrt{k})$? 
 \item  If, in addition, $h={\rm Ind}_{\mathcal{C}}$ for a compact and uniformly convex set, can we establish an improved convergence rate better than $\mathcal{O}(1/\sqrt{k})$?
\end{enumerate}

In this paper, we provide affirmative answers to the questions above.
Our contributions are as follows. 
\begin{enumerate}[\rm (i)]
\item Firstly, we propose the first single-loop CG+AL method for \eqref{problem}. In each iteration, our algorithm consists of {\em one step} of CG method applied to the AL function, followed by an update of the dual variable as in classical AL methods with a {\em diminishing} dual stepsize. As is standard with CG methods, our analysis considers both an open-loop stepsize schedule, as well as an adaptive stepsize schedule (specifically, the short stepsize schedule). This work proves that, for an open-loop stepsize schedule, the convergence rate of this CG+AL method can get arbitrarily close to $\mathcal{O}(1/\sqrt{k})$; and for the short stepsize schedule, the convergence rate is $\mathcal{O}(1/\sqrt{k})$, matching the current state-of-the-art convergence rate.

\item The second major contribution is the analysis of the important case when, additionally, $h$ is the indicator function of a uniformly convex compact set. This analysis in particular reveals improved rates for the case when the constraint set is strongly convex: (A) when using short stepsize, the primal suboptimality and feasibility violation are both bounded by $\mathcal{O}(1/k^{2/3})$, and (B) when using open-loop stepsizes, convergence rates can get arbitrarily close to $\mathcal{O}(1/k)$. The mismatch of these rates is atypical for CG methods; however, to our knowledge, such an improvement beyond $\mathcal{O}(1/\sqrt{k})$ has not been demonstrated to-date with any single-loop CG algorithm for solving \eqref{problem}.
\end{enumerate}

The rest of the paper is organized as follows. In Section~\ref{sec-pre}, we present some notation, preliminaries, the blanket assumptions on \eqref{problem}, and two important lemmas that will be used in our convergence analysis. Our CG-based single-loop AL method is presented in Section~\ref{sec-alg}, where we also establish a general convergence result under some technical assumptions. Equipped with this general convergence result, in Section~\ref{sec4}, we further
derive global convergence and complexity results under both the open-loop stepsize and short stepsize schedules. We study the case when $h={\rm Ind}_{\mathcal{C}}$ with $\mathcal{C}$ being compact and uniformly convex in Section~\ref{sec-con-ss}. Section~\ref{sec-numerics} is devoted to numerical experiments.

\section{Notation and preliminaries}\label{sec-pre}

In this paper, we use $\R^n$ to denote the $n$-dimensional Euclidean space and $\R^n_+$ to denote the nonnegative orthant of $\R^n$. Vectors in $\R^n$ will be in boldface, while scalars will use normal typeface.  For example, for ${\bm x}\in \R^n$, we write ${\bm x}=[x_1\ \cdots\ x_n]^T$ with each $x_i\in \R$. In particular, we define ${\bm 0}:=[0\ \cdots\ 0]^T$.  The $q$-norm ($q\in [1,\infty)$) of ${\bm x}$ is denoted by $\|{\bm x}\|_q=\sqrt[q]{|x_1|^q+\cdots + |x_n|^q}$; for notational simplicity, the $2$-norm of ${\bm x}$ is denoted by $\|{\bm x}\|$. We also write $[{\bm x}]_+:=\max\{{\bm x}, {\bm 0}\}$, where the maximum is taken componentwise. For ${\bm x}$ and ${\bm y }\in \R^n$, their inner product is denoted by $\langle {\bm x}, {\bm y}\rangle$. We also write ${\bm x}\geq {\bm y}$ to denote $x_i\geq y_i$ for every $i$. For a positive integer $m$, we define $[m]:=\{1,\ldots,m\}$.

Vector-valued functions are in boldface while scalar-valued functions use normal typeface; that is, we write ${\bm \phi}:=[\phi_1\ \cdots\ \phi_m]^T$ with each $\phi_i:\R^n\rightarrow \R$. 
An extended-real-valued function $\varphi : \R^n\rightarrow \R\cup\{\infty\}$ is said to be proper if its domain ${\rm dom}\, \varphi:=\{ {\bm x}\in \R^n : \varphi({\bm x})<\infty\}$ is nonempty. The function $\varphi$ is said to be closed if it is lower-semicontinuous. For a proper convex function $\varphi$, its $\epsilon$-subdifferential (for some $\epsilon\geq 0$) at an ${\bm x}\in \R^n$ is
\begin{align*}
\partial_\epsilon \varphi({\bm x}):=\{ {\bm \xi} \in \R^n : \varphi({\bm y})\geq \varphi({\bm x}) +\langle {\bm \xi}, {\bm y}-{\bm x}\rangle -\epsilon \,\,\mbox{for all ${\bm y}\in \R^n$}\};
\end{align*}
clearly, the $\epsilon$-subdifferential reduces to the subdifferential $\partial \varphi$ when $\epsilon=0$. 

For a nonempty convex set $\mathcal{C}\subseteq \R^n$,  the
indicator function is denoted by ${\rm Ind}_{\mathcal{C}}$, i.e., ${\rm Ind}_{\mathcal{C}}({\bm x})=0$ if ${\bm x}\in \mathcal{C}$ and  ${\rm Ind}_{\mathcal{C}}({\bm x})=\infty$ if ${\bm x}\notin \mathcal{C}$.
The normal cone of $\mathcal{C}$ at ${\bm x} \in \mathcal{C}$ is
\begin{align*}
\mbox{$\mathcal{N}_{\mathcal{C}}(\bm x):=\partial ({\rm Ind}_{\mathcal{C}})(\bm x)=\{ {\bm v}\in \R^n : \langle {\bm v}, {\bm y}-{\bm x} \rangle\leq 0$ \,\,for all ${\bm y}\in \mathcal{C}$\}}.
\end{align*}

\subsection{The augmented Lagrangian function}

By adding slack variable $
{\bm s}\in \R^m_+$ for the constraint ${\bm g}(\bm x)\leq {\bm 0}$, an AL function of \eqref{problem} is defined as
\begin{align}\label{lagrangian-tilde}\textstyle
\tilde{\cL}_\lambda({\bm x},{\bm s},{\bm z}):=f({\bm x})+ h({\bm x}) + {\rm Ind}_{\R^m_+}({\bm s})+\langle {\bm z},  {\bm g}({\bm x})+{\bm s}\rangle +\frac{\lambda}{2}\| {\bm g}({\bm x})+{\bm s}\|^2,
\end{align}
where ${\bm z}\in \R^m$ and $\lambda>0$ is the penalty parameter.
Minimizing $\tilde{\cL}_\lambda$ with respect to $s_i$,   the optimal $s_i$ is given by
\begin{align*}
s_i=\max\left\{ 0, - g_i(\bm x)-z_i/\lambda\right\}.
\end{align*}
  Letting $\psi_t :\R^2\rightarrow \R$ with $t>0$ be 
\begin{equation}\label{def-orig-psi}
\psi_t(u,v):= \frac{t}2\left(\left[u + \frac{v}{t}\right]_+^2 - \left(\frac{v}{t}\right)^2\right)
=
\begin{cases}
	uv +\frac{t}{2} u^2 & \mbox{if $t u+v\geq 0$}, \\
	-\frac{v^2}{2t} & \mbox{if $t u+v<0$},
\end{cases} 
\end{equation}
and substituting every optimal $s_i$ into \eqref{lagrangian-tilde} with $i\in [m]$,
we obtain the classic AL function for \eqref{problem} as follows:
\begin{align}\label{lagrangian}
\cL_\lambda({\bm x},{\bm z}):= f({\bm x})+\Psi_\lambda({\bm x},{\bm z})+h({\bm x}),
\end{align}
where $\Psi_{\lambda}:\R^n\times \R^m\rightarrow \R$ is defined as
\begin{equation}\label{def-psi}
\Psi_{\lambda}({\bm x},{\bm z}):= \sum_{i=1}^m \psi_\lambda(g_i({\bm x}), z_i) 
=  
\frac{\lambda}{2}\sum_{i=1}^m\bigg(\bigg[g_i(\bm x)+\frac{z_i}{\lambda}\bigg]_+^2-\bigg(\frac{z_i}{\lambda}\bigg)^2\bigg),
\end{equation}
which is  differentiable with respect to ${\bm x}$ and the gradient is given by
\[\textstyle
\nabla_{\bm x} \Psi_\lambda({\bm x},{\bm z})=\sum_{i=1}^m[\lambda g_i(\bm x)+z_i]_+\nabla g_i(\bm x).
\]
Denoting by $F_{\lambda}({\bm x},{\bm z}):=f({\bm x})+\Psi_\lambda({\bm x},{\bm z})$ the smooth part of $\mathcal{L}_\lambda$ in \eqref{lagrangian}, we can rewrite $\cL_\lambda$ as
\begin{align}\label{lagrangian-F}
	\cL_\lambda({\bm x},{\bm z})= F_\lambda({\bm x},{\bm z})+h({\bm x}).
\end{align}

We consider the following assumption on \eqref{problem} in this paper.
\begin{assumption}[Assumptions for functions $f$, $h$ and ${\bm g}$]\label{assu-function} Consider \eqref{problem}.
\begin{enumerate}[{\rm (i)}]
\item The function $f:\R^n\to \R$ is convex and $L_f$-Lipschitz differentiable for some $L_f > 0$.
\item  The function $h:\R^{n}\rightarrow (-\infty,+\infty]$ is proper, closed, convex, and has a bounded domain ${\rm dom}\, h$ such that
    \begin{align*}
    \sup_{{\bm x},{\bm x}'\in {\rm dom}\, h}\|{\bm x}- {\bm x}'\|\leq D<+\infty  \quad \mbox{for some positive constant $D$}.
    \end{align*}
Moreover,  for any fixed ${\bm u}\in \R^{n}$, a minimizer of the problem $\min_{{\bm v}\in \R^n} \langle  {\bm v}, {\bm u}\rangle+h(\bm v)$ exists and can be computed efficiently.
    \item For each $i\in [m]$, the function $g_i:\R^n\rightarrow \R$ is  convex differentiable and there exist positive constant $L_{g_i}$ such that
\begin{align*}
    	 \|\nabla g_i(\bm x)-\nabla g_i({\bm x}')\| \leq L_{g_i} \|{\bm x}-{\bm x}'\|  \quad 
    	 \forall \, {\bm x},{\bm x}' \in {\rm dom}\,h.
    	\end{align*}
    	\item Problem \eqref{problem} has a KKT point $({\bm x}^*,{\bm z}^*)$, i.e.,
    	\begin{equation}\label{kkt}
    		\begin{split}
    	&	\textstyle {\bm 0}\in \nabla f({\bm x}^*)+\partial h({\bm x}^*)+  \sum_{i=1}^m z_i^* \nabla g_i({\bm x}^*),   \\
    	&	{\bm g}({\bm x}^*)\leq {\bm 0}, \, {\bm z}^*\geq {\bm 0},  \, \langle {\bm z}^*, {\bm g}({\bm x}^*)\rangle= 0. \\
    		\end{split}
    	\end{equation}
\end{enumerate}
\end{assumption}

\begin{remark}\label{rem-Bi}
	Under Assumption \ref{assu-function}, for each $i\in [m]$ there exists $B_i > 0$ such that
	\begin{align*}
	  \|\nabla g_i(\bm x) \| \leq B_i  \quad \mbox{and}  \quad 	|g_i({\bm x})-g_i({\bm x}')|\leq B_i \|{\bm x}-{\bm x}'\| \quad \forall\, {\bm x},{\bm x}' \in {\rm dom}\,h,
	\end{align*}
	where the first inequality follows from the continuity of $\nabla g_i$ and the boundedness of ${\rm dom}\, h$, 
	and the second inequality follows from the mean value theorem. 
	\end{remark}

\begin{lemma}\label{lemma-pre}
		Under Assumption~\ref{assu-function}, the following statements hold.
		\begin{enumerate}[\rm (i)]
			\item {\rm (\cite[Lemma 1]{Xu-2021a}).} 
            \label{lemma-pre-1}
            The AL function $\mathcal{L}_\lambda$ given in \eqref{lagrangian} is convex with respect to ${\bm x}$ and concave with respect to ${\bm z}$. 
		\item {\rm (\cite[Lemma 1]{Xu-2021b}).}
         \label{lemma-pre-2}
         For any $({\bm x},{\bm z})$ with ${\bm g}({\bm x})\leq {\bm 0}$, it holds that $\Psi_\lambda({\bm x},{\bm z})\leq 0$.
			\item 
             \label{lemma-pre-3}
             For any KKT point  $({\bm x}^*,{\bm z}^*)$ of \eqref{problem}, it holds that $\Psi_\lambda({\bm x}^*,{\bm z}^*)= 0$.
			\item {\rm (\cite[Lemma 2]{Xu-2021b}).} 
             \label{lemma-pre-4}
             For any fixed ${\bm z}$, the function $\Psi_\lambda(\cdot,{\bm z})$ is continuously differentiable and satisfies that for any ${\bm x}$, ${\bm x}'\in {\rm dom}\, h$
			\begin{equation}\label{lip-psi}
	\|\nabla_{\bm x} \Psi_\lambda({\bm x},{\bm z}) -  \nabla_{\bm x} \Psi_\lambda({\bm x}',{\bm z})\| 
		\leq \underbrace{\sum_{i=1}^m(\lambda B_i^2 + L_{g_i} [\lambda g_i({\bm x})+z_i]_+)}_{=:L_{\Psi}({\bm x},{\bm z})}\|{\bm x}-{\bm x}'\|,
			\end{equation}
			where $B_i$ and $L_{g_i}$ are given in  Remark \ref{rem-Bi} and  Assumption \ref{assu-function}, respectively.
            In addition, it holds that 
            \begin{equation}\label{lip-bound}
            L_{\Psi}({\bm x},{\bm z})=\mathcal{O}(\max\{\lambda,\|\bm z\|\}).
            \end{equation}
			\end{enumerate}
		\end{lemma}
		\begin{proof}
			Item (iii) follows directly from the definition of $\Psi_\lambda$ in \eqref{def-psi} and the fact that $({\bm x}^*,{\bm z}^*)$ satisfies \eqref{kkt}.
          To obtain the last assertion in item  (iv), notice that 
            \[\textstyle
            L_{\Psi}({\bm x},{\bm z})\leq\lambda	\sum_{i=1}^m (B_i^2+L_{g_i} |g_i({\bm x})|)+ \|{\bm z}\| \sqrt{m}\max_i L_{g_i}=\mathcal{O}(\max\{\lambda,\|\bm z\|\}),
            \]
            where the equality follows from the boundedness of ${\rm dom}\, h$ and the continuity of ${\bm g}$.
			\end{proof}

\subsection{Two difference inequalities}

In this subsection, we analyze two difference inequalities, which are the central tools for our convergence analysis. Analogous inequalities have been studied in the literature; see \cite[Lemma 4.4]{Bach-2015} and \cite[Lemma A.1]{Kerdreux-dAspremont-Pokutta-2021}. The difference inequality \eqref{prop_phi1-1} we study below is different from the one in \cite[Lemma 4.4]{Bach-2015} as they only considered $(\tau_k,\beta_k)=(\Theta(1/k),\Theta(1/k^2))$, while the difference inequality \eqref{phi_relation1} we study below differs from the one in \cite[Lemma A.1]{Kerdreux-dAspremont-Pokutta-2021} since we have an additional term of $\gamma_k$ and consider, more generally, an $\eta\in [0.5,1)$. The inequalities \eqref{prop_phi1-1} and \eqref{phi_relation1} are instrumental for our convergence analysis under the open-loop stepsize and the short stepsize, respectively.

\begin{proposition}\label{prop_phi1}
    Let $\{\tau_k\}$ and $\{\beta_k\}$ be  positive nonincreasing sequences satisfying $\tau_k =\Theta(k^{-t_1})$ and $\beta_k =\Theta(k^{-t_2})$ for some $t_1\in (0,1)$ and $t_2>t_1$. If $\{\phi_k\}$ is a nonnegative sequence satisfying 
    \begin{align}\label{prop_phi1-1}
    \phi_{k+1}\le (1-\tau_k)\phi_k + \beta_k
    \end{align}
    for all $k\ge 0$, then it holds that $\phi_k = \mathcal{O}(k^{-(t_2-t_1)})$.
\end{proposition}
\begin{proof}
As $\beta_k =\Theta(k^{-t_2})$, we have $\beta_k\leq {c_1/(k+1)^{t_2}}$ for some $c_1>0$. 
Observe that
\begin{align*}
(1-\tau_k)(1+1/(k+1))^{t_2}=(1-\tau_k)(1+t_2/(k+1)+o(k^{-1}))=1-\tau_k +\mathcal{O}(k^{-1}),
\end{align*}
where the first equality follows from the fact that $(1+x)^{t_2} = 1+t_2 x + o(x)$ for all sufficiently small $x>0$.
The above display together with $\tau_k=\Theta(k^{-t_1})$ with $t_1\in (0,1)$ implies that there exists a positive integer $k_0$ such that for any $k\geq k_0$,
\begin{align}\label{prop_phil-1}
0< (1-\tau_k)(1+1/(k+1))^{t_2}\leq 1-0.5\tau_k \quad \mbox{and} \quad (1+1/(k+1))^{t_2}\leq 1+2 t_2,
\end{align}
where for the second relation we use the fact that $(1+x)^{t_2} \leq 1 +2 t_2 x$ for all sufficiently small $x>0$.
Let $\hat{\phi}_k:=\phi_k (k+1)^{t_2} \tau_k$.  
It follows that for any $k\geq k_0$
\begin{align}
\hat{\phi}_{k+1}=&\phi_{k+1} (k+2)^{t_2}\tau_{k+1} \overset{\rm (a)}{\leq} [(1-\tau_k)\phi_k + c_1(k+1)^{-{t_2}}](k+2)^{t_2}\tau_{k+1} \notag \\
= & (\phi_k(k+1)^{t_2}\tau_k) (1-\tau_k)(1+1/(k+1))^{t_2} (\tau_{k+1}/\tau_k)\notag \\
&+c_1(1+1/(k+1))^{t_2}(\tau_{k+1}/\tau_k)\tau_{k} \notag\\
\overset{\rm (b)}{\leq} & \hat{\phi}_k(1-\tau_k)(1+1/(k+1))^{t_2}  +c_1 (1+1/(k+1))^{t_2} \tau_k\notag \\
\overset{\rm (c)}{\leq} &\hat{\phi}_k(1-0.5\tau_k) + c_1(1+2t_2) \tau_k \overset{\rm (d)}{=}\hat{\phi}_k (1-\hat{\tau}_k)+\varpi \hat{\tau_k},\label{prop-phi1-2}
\end{align}
where (a) holds because $\phi_{k+1}\le (1-\tau_k)\phi_k + \beta_k$ and 
 $\beta_k\leq {c_1/(k+1)^{t_2}}$, 
 (b) follows from the definition of $\hat{\phi}_k$ and the facts that $\tau_k$ is nonincreasing and $\tau_k < 1$ (see \eqref{prop_phil-1}),
 (c) follows from \eqref{prop_phil-1},
 and (d) follows upon setting $\hat{\tau}_k:=0.5\tau_k$ and $\varpi:=2c_1(1+2t_2)$.

Since $\hat\tau_k =0.5\tau_k$, from the first inequality in \eqref{prop_phil-1}, we have that $1-\hat\tau_k> 0$ for all $k\ge k_0$.
We will finish the proof by showing that 
    $\sup_{k\ge k_0}\hat{\phi}_k \le \max\{\hat{\phi}_{k_0},\varpi\} =: M$.  
    We proceed by induction. Clearly, $\hat{\phi}_{k_0}\leq M$. Now suppose that $\hat{\phi}_l\le M$ for some $l\ge k_0$. Since $1-\hat{\tau}_l> 0$, we have from \eqref{prop-phi1-2}
    \[
    \hat{\phi}_{l+1} \le (1-\hat{\tau}_l)\hat{\phi}_l + \varpi\hat{\tau}_l \le (1-\hat{\tau}_l)M + \varpi \hat{\tau}_l \le (1-\hat{\tau}_l)M +  \hat{\tau}_l M = M,
    \]
    where the first inequality follows from \eqref{prop-phi1-2}. 
    This completes the induction argument. The desired conclusion now follows from the boundedness of $\{\hat{\phi}_k\}$.
\end{proof}

\begin{proposition}\label{prop_phi2}
    Let $\{\gamma_k\}$ be a positive nonincreasing sequence, $\eta\in[0.5,1)$ and $\mu\in (0,1]$. Suppose that
    \begin{equation}\label{condition_for_gamma}
       \lim_{k\to \infty}\gamma _k = 0\ \ {and}\ \ \sup_{k} \gamma_k^{-1}(\gamma^{1/(1+\mu)}_{k} - \gamma^{1/(1+\mu)}_{k+1}) < \infty.
    \end{equation}
    If $\{\phi_k\}$ is a nonnegative sequence satisfying
    \begin{equation}\label{phi_relation1}
\phi_{k+1}\le \phi_k \max\{\eta,1-\phi_k^\mu\} + \gamma_k\ \ \ \ \forall k\ge 0,
    \end{equation}
    then $\phi_k = \mathcal{O}(\gamma_k^{1/(1+\mu)})$.
\end{proposition}
\begin{remark}\label{comment_condition_gen}
    We comment on the condition \eqref{condition_for_gamma}. Suppose that $\gamma_k = c (k+1)^{-s}$ for some $c>0$ and $s > 0$. Then clearly $\{\gamma_k\}$ is positive nonincreasing and $\gamma_k \to 0$. We claim that the second relation in \eqref{condition_for_gamma} holds if $\mu \in (0,1]$ satisfies $s \le 1 + 1/\mu$.

    To see this, notice that the function $t\mapsto t^{1/(1+\mu)}$ is concave for $t\ge 0$. Hence,
    \begin{align*}
        &0\le \gamma_k^{-1}(\gamma_k^{1/(1+\mu)}-\gamma_{k+1}^{1/(1+\mu)})\le (1+\mu)^{-1}\gamma_k^{-1}\left(\gamma_{k+1}^{-\mu/(1+\mu)}\left(\gamma_k-\gamma_{k+1}\right)\right)\\
        &= \mathcal{O}(k^{s} k^{s\mu/(1+\mu)}((k+1)^{-s} \!-\! (k+2)^{-s}))\!= \! \mathcal{O}(k^{s} k^{s\mu/(1+\mu)}(k+2)^{-s}((1+1/(k+1))^s\!-\!1))\\
        &=\mathcal{O}(k^{-1+s\mu/(1+\mu)})=\mathcal{O}(k^{-1+s/(1+1/\mu)}).
    \end{align*}
    Since $s \in (0, 1 + 1/\mu]$, the above display shows that the second relation in \eqref{condition_for_gamma} holds.
\end{remark}
\begin{proof}[Proof of Proposition~\ref{prop_phi2}]
   {\bf Step 1}: We first show that there exists $k_0 \ge 1$ with
   \begin{equation}\label{claim1}
      \phi_{k_0}\le (1-\eta)^{1/\mu}\ \ \ {\rm and}\ \ \ \gamma_{k_0}\le \eta (1-\eta)^{1+1/\mu}.
   \end{equation}
   To see that, notice first that in view of the limit in \eqref{condition_for_gamma}, we can choose $k_1 \ge 1$ such that $\gamma_{k_1} \le \eta (1-\eta)^{1+1/\mu}$. Since $\{\gamma_k\}$ is nonincreasing, we see that $\gamma_{k} \le \eta (1-\eta)^{1+1/\mu}$ for all $k\ge k_1$. Now, suppose to the contrary that $\phi_{k}> (1-\eta)^{1/\mu}$ for all $k\ge k_1$. Then we see from \eqref{phi_relation1} that for all $k\ge k_1$,
   \begin{align*}
       \phi_{k+1}&\le \phi_k \max\{\eta,1-\phi_k^\mu\} + \gamma_k 
       \overset{\rm (a)}\le \eta\phi_k +\eta(1-\eta)^{1+1/\mu}\\
       & = \eta \phi_k +  (1-\eta)^{1+1/\mu} - (1-\eta)^{2+1/\mu} \overset{\rm (b)}{<} \phi_k -  (1-\eta)^{2+1/\mu},
   \end{align*}
   where (a) holds because $\gamma_{k} \le \eta (1-\eta)^{1+1/\mu}$ for all $k\ge k_1$,
   and (b) holds because $\phi_{k}> (1-\eta)^{1/\mu}$ for all $k\ge k_1$. Since  $(1-\eta)^{2+1/\mu}$ is a positive constant,
    the above display contradicts the nonnegativity of $\{\phi_k\}$. Thus, there exists $k_0\ge k_1 \ge 1$ such that \eqref{claim1} holds.

   {\bf Step 2}: For the $k_0$ obtained above, we now show that $\phi_{k}\le (1-\eta)^{1/\mu}$ for all $k\ge k_0$. We proceed by induction. We have $\phi_{k_0}\le (1-\eta)^{1/\mu}$ from {\bf Step 1} (see \eqref{claim1}). Suppose that $\phi_{t}\le (1-\eta)^{1/\mu}$ for some $t\ge k_0$. Then we have
   \begin{align*}
       \phi_{t+1}&\le \phi_t \max\{\eta,1-\phi_t^\mu\} + \gamma_t 
       \overset{\rm (a)}\le  (1-\eta)^{1/\mu} \max\{\eta,1-((1-\eta)^{1/\mu})^\mu\} + \gamma_{k_0} \\
       & \overset{\rm (b)}\le  \eta (1-\eta)^{1/\mu} + \eta(1-\eta)(1-\eta)^{1/\mu}
       =(2\eta-\eta^2) (1-\eta)^{1/\mu} \overset{\rm (c)}{<} (1-\eta)^{1/\mu},
   \end{align*}
   where (a) holds because the function $x\mapsto x\max\{\eta,1-x^\mu\}$ is increasing on $\R_+$ for any $\eta\in[0.5,1)$ and $\{\gamma_k\}$ is nonincreasing, (b) holds because $\gamma_{k_0}\le \eta(1-\eta)^{1+1/\mu}$ (see \eqref{claim1}),
   and (c) follows from the fact that $\eta\in[0.5,1)$ and $2x-x^2 \leq 1$  (the equality holds only for $x=1$).
   This completes the induction argument.

   {\bf Step 3}: From the above two steps and \eqref{phi_relation1}, we see that
   \begin{equation}\label{phi_relation11}
\phi_{k+1}\le \phi_k - \phi_k^{1+\mu} + \gamma_k\ \ \ {\rm and}\ \ \ \phi_k \le (1-\eta)^{1/\mu}\ \ \ \ \forall k\ge k_0.
    \end{equation}
    Now, from \eqref{condition_for_gamma}, we have
    $
    \sup_{k}\gamma_{k}^{-1}(\gamma^{1/(1+\mu)}_{k} - \gamma^{1/(1+\mu)}_{k+1}) < \infty$ 
    and 
    $
    \lim_{k\to \infty}\gamma_k^{-1/(1+\mu)} = \infty
    $.
    Consequently, there exists $k_2 \ge k_0$ such that
    \begin{align}
        &C_1 := (1-\eta)^{1/\mu} \gamma_{k_2}^{-1/(1+\mu)} > 1,\label{Crel1}\\
        &\sup_{k}(\gamma^{1/(1+\mu)}_{k} - \gamma^{1/(1+\mu)}_{k+1})\gamma_k^{-1} \le C_1^\mu - 1.\label{Crel2}
    \end{align}

    We will now finish the proof by showing that $\phi_k \le C_1 \gamma_k^{1/(1+\mu)}$ for all $k\ge k_2$.     
    We proceed by induction. First, since $k_2 \ge k_0$, we see from {\bf Step 2} that
    \[
    \phi_{k_2} \le (1-\eta)^{1/\mu} = (1-\eta)^{1/\mu} \gamma_{k_2}^{-1/(1+\mu)}\gamma_{k_2}^{1/(1+\mu)}= C_1 \gamma_{k_2}^{1/(1+\mu)},
    \]
    where the equality follows from the definition of $C_1$ in \eqref{Crel1}. Now, suppose that $\phi_t \le C_1 \gamma_t^{1/(1+\mu)}$ for some $t\ge k_2$. Then we have $C_1\gamma_t^{1/(1+\mu)} \le C_1\gamma_{k_2}^{1/(1+\mu)} = (1-\eta)^{1/\mu}$. Using this, we can deduce from \eqref{phi_relation11} that
    \begin{align*}
        \phi_{t+1}&\le \phi_t - \phi_t^{1+\mu} + \gamma_t\le C_1\gamma_t^{1/(1+\mu)} - C_1^{1+\mu}\gamma_t + C_1\gamma_t\\
        & = C_1 \gamma_{t+1}^{1/(1+\mu)} + C_1\gamma_t\big(\gamma_t^{-1}(\gamma_t^{1/(1+\mu)} - \gamma_{t+1}^{1/(1+\mu)}) - C_1^{\mu} + 1\big)\le C_1 \gamma_{t+1}^{1/(1+\mu)},
    \end{align*}
    where in the second inequality we used the facts that $x\mapsto x-x^{1+\mu}$ is increasing on $[0,0.5^{1/\mu}]$ (and noting that $\phi_t \le C_1\gamma_t^{1/(1+\mu)} \leq (1-\eta)^{1/\mu}\leq 0.5^{1/\mu}$ since $\eta\in[0.5,1)$) and $C_1 > 1$ (see \eqref{Crel1}), and the last inequality follows from \eqref{Crel2}. This completes the induction argument and finishes the proof.
\end{proof}

\section{A CG-based single-loop AL method}\label{sec-alg}

Algorithm~\ref{alg} contains our new single-loop CG+AL method that also accounts for the inequality constraints in \eqref{problem} under Assumption~\ref{assu-function}.
The operations in \eqref{alg-1} and \eqref{alg-2} correspond to  {\em one-step} of the conditional gradient method to minimize the AL function ${\cal L}_{\lambda_k}({\bm x},{\bm z}^k)=F_\lambda({\bm x},{\bm z}^k)+h({\bm x})$, then \eqref{alg-3} updates the dual iterate ${\bm z}^k$ as in the classical AL method with a diminishing dual stepsize. Note that Algorithm~\ref{alg} allows $\alpha_k = 0$ to cover the use of a generalized short stepsize (see \eqref{short_stepsize_haha} below) while requiring $\{\lambda_k\}$ to be increasing so that ${\cal L}_{\lambda_k}$ is updated every iteration. 


\begin{algorithm}[h]
\caption{A CG-based single-loop AL method for \eqref{problem} under Assumption~\ref{assu-function}} \label{alg}
\begin{enumerate}[\textbf{Step} 1.]
  \item  Choose  $({\bm x}^0,{\bm z}^0)\in   {\rm dom}\, h \times \R^m_+$, $\alpha_0 \in [0,1]$,
  a positive increasing sequence $\{\lambda_k\}$ (i.e., $\lambda_{k+1} > \lambda_k > 0$ for all $k$) with $\lim_{k\to\infty}\lambda_k = \infty$, and a positive summable sequence $\{\sigma_k\}$ with $\sigma_k\leq \lambda_k$.
  Set $k=0$. 

\item Compute
\begin{align}
& {\bm v}^{k} \in\Argmin_{{\bm v}\in \R^n}  \langle  \nabla_{\bm x} F_{\lambda_k}({\bm x}^k,{\bm z}^k),{\bm v}\rangle + h({\bm v}),\label{alg-1}\\
& {\bm x}^{k+1} ={\bm x}^k+\alpha_k ({\bm v}^k-{\bm x}^k), \label{alg-2} \\
&{\bm z}^{k+1}={\bm z}^k + \sigma_k \max\left\{-\frac{{\bm z}^k}{\lambda_k}, \, {\bm g}({\bm x}^{k+1})\right\}. \label{alg-3}
\end{align}
\item
  Choose $\alpha_{k+1}\in [0,1]$ and update $k\leftarrow k+1$. Go to Step 2.
\end{enumerate}
\end{algorithm}

\begin{remark}[Boundedness for sequences generated by Algorithm \ref{alg}]\label{remark-bound}
	Suppose that Assumption~\ref{assu-function} holds and let $\{({\bm x}^k,{\bm z}^k)\}$ be generated by Algorithm \ref{alg}.
	Since ${\rm dom}\,h$ is bounded and $\{\alpha_k\}\subset[0,1]$, we see from \eqref{alg-1} and \eqref{alg-2} that $\{{\bm x}^k\}$ is  bounded.
	
	In addition, from Step 1 of Algorithm~\ref{alg}, we have $\sigma_k\leq \lambda_k$ and ${\bm z}^0\geq {\bm 0}$. Thus, it follows from \eqref{alg-3} that ${\bm z}^k \geq {\bm 0}$ for all $k$.
	
Next, we show the boundedness of $\{{\bm z}^k\}$.
For each $i\in[m]$,	if $g_i({\bm x}^{k+1})< -z_i^k/\lambda_k$, we see that
	$z_i^{k+1}=(1-\sigma_k/\lambda_k)z_i^k \leq z_i^k$ thanks to the facts that $\sigma_k\leq \lambda_k$ and ${\bm z}^k\geq {\bm 0}$; if $g_i({\bm x}^{k+1})\geq  -z_i^k/\lambda_k$, in view of \eqref{alg-3}, we see  that
	\begin{align}\label{remark-bound-zk-0}
	z_i^{k+1} \!=\! z_i^k\!+\! \sigma_k g_i({\bm x}^{k+1})\!\leq\! z_i^k \!+\!\sigma_k [g_i({\bm x}^{k+1})]_+   \!\leq\! z_i^k \!+\!\sigma_k \|[{\bm g}({\bm x}^{k+1})]_+\| \!\leq\! z_i^k\!+\! M_g\sigma_k,
	\end{align}
    where in the last inequality we set $M_g:=\sup\{\|[{\bm g}({\bm x})]_+\|: {\bm x}\in {\rm dom}\,h\}<\infty$ (this quantity is finite thanks to the continuity of ${\bm g}$ on $\R^n$ and the boundedness of ${\rm dom}\, h$ by Assumption~\ref{assu-function}).
	Then, from the above discussions and the summability of $\{\sigma_k\}$ (see Step 1 of Algorithm~\ref{alg}),  we obtain the boundedness of $\{{\bm z}^k\}$ as follows
	\begin{equation}\label{remark-bound-zk}
    \textstyle
		\|{\bm z}^{k}\| = \big(\sum_{i=1}^m |z^k_i|^2\big)^{1/2}\leq \sum_{i=1}^m |z_i^k|\le
        \|{\bm z}^0\|_1 + M_g\sum_{j=0}^\infty\sigma_j<\infty.
	\end{equation}
    Finally, we conclude this remark by noting the existence of $\varsigma>0$ such that
    \begin{align}\label{Lip-Lpsi}
    L_{\Psi}({\bm x}^k,{\bm z}^k)\leq \varsigma \lambda_k \ \ \ \forall k,
    \end{align}
    which is an immediate consequence of the boundedness of $\{{\bm z}^k\}$ and \eqref{lip-bound} established above.
\end{remark}


For the rest of this section, we will derive some useful properties of Algorithm~\ref{alg} for our subsequent analysis, and present a general convergence result under some technical assumptions: these technical assumptions will be shown to hold by suitably choosing the algorithm parameters (such as $\{\alpha_k\}$) in Section~\ref{sec4}.

We start by defining the following auxiliary gap function for \eqref{problem} under Assumption~\ref{assu-function}. Given $({\bm x},{\bm z})\in {\rm dom}\, h \times \R^m_+$ and $\lambda>0$, define 
\begin{align}\label{def-G}
	G_{{\bm z},\lambda}({\bm x}):=\langle \nabla_{\bm x} F_\lambda({\bm x},{\bm z}), {\bm x}-{\bm v}^+\rangle + h({\bm x}) - h({\bm v}^+),
\end{align}
where ${\bm v}^+\in \Argmin_{{\bm v}}   \langle \nabla_{\bm x} F_{\lambda}({\bm x}, {\bm z}),{\bm v}\rangle + h({\bm v})$, and one can observe that the above definition is {\em independent} of the choice of ${\bm v}^+$ in $\Argmin_{{\bm v}}   \langle \nabla_{\bm x} F_{\lambda}({\bm x}, {\bm z}),{\bm v}\rangle + h({\bm v})$ and that $G_{{\bm z}, \lambda}(\bm x)\ge 0$. When ${\bm g}\equiv {\bm 0}$, $G_{{\bm z},\lambda}$ reduces to the standard gap function considered in the Frank-Wolfe literature for minimizing $f + h$, where $G_{{\bm z},\lambda}({\bm x})=0$ means ${\bm x}$ is globally optimal; see, e.g., \cite{Jaggi-2013,Beck-Pauwels-Sabach-2015, Freu16, Pena-2023}. However, when ${\bm g}\not\equiv {\bm 0}$, $G_{{\bm z},\lambda}({\bm x})=0$ only guarantees that ${\bm x}$ is a minimizer of \eqref{lagrangian-F} for fixed values of ${\bm z}$ and $\lambda$; in particular, this does not necessarily imply that ${\bm x}$ is optimal for \eqref{problem} because ${\bm x}$ may not satisfy ${\bm g}({\bm x})\le {\bm 0}$.

The next lemma relates the gap function to the AL function.
\begin{lemma}[Lower bounds for $G_{{\bm z}, \lambda}({\bm x})$]\label{lemma-G-T}
Suppose that Assumption~\ref{assu-function} holds.
For any $({\bm x},{\bm z})\in {\rm dom}\, h \times \R^m_+$ and $\lambda>0$ it holds that 
\[
G_{{\bm z}, \lambda}({\bm x})\geq [\mathcal{L}_\lambda({\bm x}, {\bm z})-L^*]_+,
\]
where $G_{{\bm z}, \lambda}$ is given in \eqref{def-G}, $\cL_\lambda$ is defined in \eqref{lagrangian} and $L^*$ is the optimal value of \eqref{problem}.
\end{lemma}
\begin{proof}
Let $({\bm x}^*,{\bm z}^*)$ be the KKT point given in Assumption~\ref{assu-function}(iv). Then ${\bm x}^*$ is a solution of \eqref{problem}, and we have for any ${\bm v}^+\in \Argmin_{{\bm v}}   \langle \nabla_{\bm x} F_{\lambda}({\bm x}, {\bm z}),{\bm v}\rangle + h({\bm v})$ that
\begin{align}
	G_{{\bm z}, \lambda}({\bm x})& =\langle \nabla_{\bm x} F_{\lambda}({\bm x},{\bm z}), {\bm x}-{\bm v}^+\rangle+ h({\bm x})-h({\bm v}^+) \geq \langle \nabla_{\bm x} F_{\lambda}({\bm x},{\bm z}), {\bm x}-{\bm x}^*\rangle+ h({\bm x})-h({\bm x}^*) \notag \\
	&\overset{\rm (a)}\geq F_{\lambda}({\bm x},{\bm z})- F_{\lambda}({\bm x}^*,{\bm z})
    + h({\bm x})-h({\bm x}^*)
	\overset{\rm (b)}\geq F_{\lambda}({\bm x},{\bm z})-F_{\lambda}({\bm x}^*,{\bm z}^*)+ h({\bm x})-h({\bm x}^*)\notag\\
    &=\mathcal{L}_{\lambda}({\bm x},{\bm z})-L^*, \notag
\end{align}
where (a) uses convexity of $F_\lambda(\cdot,{\bm z})$, and (b) follows from Lemma \ref{lemma-pre},
and the last equality follows from \eqref{lagrangian-F}. 
Invoking the nonnegativity of $G_{{\bm z}, \lambda}(\bm x)$ completes the proof.
\end{proof}

When ${\bm g}\equiv {\bf 0}$, the gap function appears naturally in the short stepsize schedule. With the gap function \eqref{def-G} in mind, one can naturally define the analogous short stepsize for the AL approach to solving \eqref{problem} under Assumption~\ref{assu-function} as
\begin{equation}\label{short_stepsize_haha}
\alpha^{\rm short} := \begin{cases}
0 & {\rm if }\ {G_{{\bm z},\lambda}({\bm x})}=\|{\bm v}^+-{\bm x}\|=0,\\
\displaystyle \min\left\{1,\frac{G_{{\bm z},\lambda}({\bm x})}{(L_\Psi({\bm x},{\bm z})+L_f)\|{\bm v}^+-{\bm x}\|^2}\right\} & {\rm otherwise},
\end{cases}
\end{equation}
where $G_{{\bm z},\lambda}({\bm x})$ and ${\bm v}^+$ are given in \eqref{def-G}, $L_{\Psi}$ is given in \eqref{lip-psi}, and $L_f$ is given in Assumption~\ref{assu-function}(i). Note that unlike the classical short stepsize for the case ${\bm g}\equiv 0$, the above short stepsize can be zero when ${\bm x}$ is not yet optimal for \eqref{problem}, precisely because $G_{{\bm z},\lambda}({\bm x})=0$ is not sufficient for guaranteeing optimality, as explained above.

Next, we study the changes to $\Psi_\lambda$ in \eqref{def-psi} when $\lambda$ and the dual variable ${\bm z}$ are updated under two stepsize rules: general stepsizes in $[0,1]$ and the short stepsize \eqref{short_stepsize_haha}.

\begin{lemma}[Change in $\Psi$ w.r.t. ${\bm z}$ and $\lambda$]\label{lemma-psi}
Suppose that Assumption~\ref{assu-function} holds.
	Let $\{({\bm x}^k,{\bm z}^k,{\bm v}^k)\}$ be generated by Algorithm \ref{alg} and $\Psi_{\lambda}$ be given in \eqref{def-psi}.
	Then, there exists a positive constant $c$ such that for all $k\ge 0$, 
	\begin{align}
		&\Psi_{\lambda_{k+1}}({\bm x}^{k+1},{\bm z}^{k+1})-\Psi_{\lambda_k}({\bm x}^{k+1},{\bm z}^k) \notag\\
		\leq &  c \max\{1/\lambda_k^2, \alpha^2_k,  \|[{\bm g}({\bm x}^k)]_+\|^2\}(\sigma_k + \lambda_{k+1}-\lambda_k). \label{lemma-psi-1}
	\end{align}
    If $\alpha_k$ is the short stepsize in \eqref{short_stepsize_haha} with $({\bm x}^k,{\bm z}^k,{\bm v}^k,\lambda_k)$ in place of $({\bm x},{\bm z},{\bm v}^+,\lambda)$, then it holds
    	\begin{align}
		&\Psi_{\lambda_{k+1}}({\bm x}^{k+1},{\bm z}^{k+1})-\Psi_{\lambda_k}({\bm x}^{k+1},{\bm z}^k) \notag\\
		\leq &  c \max\{1/\lambda_k^2, \alpha_k G_{{\bm z}^k,\lambda_k}({\bm x}^k),  \|[{\bm g}({\bm x}^k)]_+\|^2\}(\sigma_k + \lambda_{k+1}-\lambda_k). \label{lemma-psi-2}
	\end{align}
\end{lemma}
\begin{proof}
    From  \eqref{alg-3} in Algorithm \ref{alg}, it holds that for each $i\in [m]$
		\begin{align}
	|z_i^{k+1}-z_i^k|^2=\sigma^2_k |\max\{-z_i^k/\lambda_k,  g_i({\bm x}^{k+1})\}|^2 \leq  \sigma^2_k \max\{(z_i^k/\lambda_k)^2,  [g_i({\bm x}^{k+1})]_+^2\}. \label{lemma-psi-3}
	\end{align}
    Also, notice that for each $i\in[m]$
    \begin{align}
     [g_i({\bm x}^{k+1})]_+&\leq  [g_i({\bm x}^{k+1})-g_i({\bm x}^k)]_+ +[g_i({\bm x}^k)]_+ \notag  \\
     &\leq  B_i \alpha_k\|{\bm v}^k-{\bm x}^k\|+[g_i({\bm x}^k)]_+ \label{lemma-psi-3-00} \\
   &\leq  B_i D \alpha_k +[g_i({\bm x}^k)]_+,\label{lemma-psi-3-0}
    \end{align}
    where the second inequality follows from Remark~\ref{rem-Bi} and \eqref{alg-2},
    and the last inequality follows from Assumption~\ref{assu-function}(ii).
    From \eqref{lemma-psi-3-00}, if $\alpha_k$ is the short stepsize we have 
    \begin{align}
      [g_i({\bm x}^{k+1})]_+^2\leq & 2B_i^2 \alpha_k^2 \|{\bm v}^k-{\bm x}^k\|^2 +2 [g_i({\bm x}^k)]_+^2 \notag\\
      =  &2B_i^2 \alpha_k^2 \|{\bm v}^k-{\bm x}^k\|^2 (L_{\Psi}({\bm x}^k,{\bm z}^k)+L_f) (L_{\Psi}({\bm x}^k,{\bm z}^k)+L_f)^{-1} +2 [g_i({\bm x}^k)]_+^2 \notag \\
      \leq & 2B_i^2 L_f^{-1} \alpha_k G_{{\bm z}^k,\lambda_k}({\bm x}^k)+2 [g_i({\bm x}^k)]_+^2,  \label{lemma-psi-3-000}
    \end{align}
    where the last inequality follows by a direct calculation using the definition of the short stepsize and the fact that $(L_{\Psi}({\bm x}^k,{\bm z}^k)+L_f)^{-1}\leq L_f^{-1}$.

	Next, in view of \eqref{lemma-psi-3-0} and \eqref{lemma-psi-3-000},  the fact that $\Psi_{\lambda}({\bm x},{\bm z})= \sum_{i=1}^m \psi_\lambda(g_i({\bm x}), z_i) $ (see \eqref{def-psi}) and the boundedness of $\{{\bm z}^k\}$ (see \eqref{remark-bound-zk}), to prove \eqref{lemma-psi-1} and \eqref{lemma-psi-2}, it suffices to show for each $i\in [m]$
    \begin{align}
&\psi_{\lambda_{k+1}}(g_i({\bm x}^{k+1}),z_i^{k+1})-\psi_{\lambda_k}(g_i({\bm x}^{k+1}),z_i^k) \notag\\
&=\mathcal{O}(\max\{(z_i^k/\lambda_k)^2,  [g_i({\bm x}^{k+1})]_+^2\}[\sigma_k+(\lambda_{k+1}-\lambda_k)]). \label{lemma-psi-3-1}
    \end{align} 
To this end, we consider different cases for each $i\in [m]$ as follows.
	\begin{enumerate}[\rm (I)]
		\item $\lambda_{k+1} g_i({\bm x}^{k+1}) + z_i^{k+1}\geq 0$ and   $\lambda_k g_i({\bm x}^{k+1})+ z_i^k\geq 0$.
		
		In this case, it holds that 
        \begin{equation}\label{z_update_haha}
        z_i^{k+1}=z_i^k+\sigma_k g_i({\bm x}^{k+1})
        \end{equation}
        thanks to \eqref{alg-3} and the fact that $g_i({\bm x}^{k+1})\geq -z_i^k/\lambda_k$.
		Then we see that
			\begin{align}
			&\psi_{\lambda_{k+1}}(g_i({\bm x}^{k+1}),z_i^{k+1})-\psi_{\lambda_k}(g_i({\bm x}^{k+1}),z_i^k)   \notag  \\
			=&g_i({\bm x}^{k+1})z_i^{k+1} +\frac{\lambda_{k+1}}{2} g_i({\bm x}^{k+1})^2-g_i({\bm x}^{k+1})z_i^k -\frac{\lambda_{k}}{2} g_i({\bm x}^{k+1})^2 \notag \\
			\leq & \sigma_k^{-1} |z_i^{k+1}-z_i^k|^2+\frac{\lambda_{k+1}-\lambda_k}{2} \max\{ (z^k_i/\lambda_k)^2, [g_i({\bm x}^{k+1})]_+^2\}, \notag
		\end{align}
        where the equality follows from \eqref{def-orig-psi}, and
        the inequality holds thanks to \eqref{z_update_haha}, the monotonicity of $\{\lambda_k\}$ and  the fact that $g_i({\bm x}^{k+1})\geq -z_i^k/\lambda_k$.
Using the above display together with \eqref{lemma-psi-3}, we see that \eqref{lemma-psi-3-1} holds.
			
		\item 		 $\lambda_{k+1} g_i({\bm x}^{k+1}) + z_i^{k+1}\geq 0$ and   $\lambda_k g_i({\bm x}^{k+1})+ z_i^k< 0$.

We note that this case is void.  To see this,
notice that $\lambda_k g_i({\bm x}^{k+1})+z_i^k<0$ means $g_i({\bm x}^{k+1})<-z_i^k/\lambda_k$, which together with \eqref{alg-3} implies that
		$z_i^{k+1}=z_i^k- (\sigma_k/\lambda_k) z_i^k$. 
		Since $z_i^k\geq 0$ (see Remark \ref{remark-bound}) and $0 < \sigma_k\le \lambda_k$, we see that $z_i^k\geq z_i^{k+1} \geq 0$.
		Also note that $g_i({\bm x}^{k+1})<0$, and that $\{\lambda_k\}$ is increasing. We must then have $\lambda_{k+1}g_i({\bm x}^{k+1})+z_i^{k+1} \le \lambda_k g_i({\bm x}^{k+1})+z_i^k<0$.
			
		\item $\lambda_{k+1} g_i({\bm x}^{k+1}) + z_i^{k+1}< 0$ and   $\lambda_k g_i({\bm x}^{k+1})+ z_i^k< 0$.
		
			In this case, we have $g_i({\bm x}^{k+1})<0$ since $z_i^k\geq 0$. In addition, since $g_i({\bm x}^{k+1})<-z_i^k/\lambda_k$, we have from \eqref{alg-3} that 
			\begin{align} \label{lemma-psi-4}
			z_i^{k+1}=z_i^k-\sigma_kz_i^k/\lambda_k.
			\end{align}
			We  deduce that
		\begin{align}
	&\psi_{\lambda_{k+1}}(g_i({\bm x}^{k+1}),z_i^{k+1})-\psi_{\lambda_k}(g_i({\bm x}^{k+1}),z_i^k)   
			= \frac{|z_i^k|^2}{2\lambda_k} - \frac{|z_i^{k+1}|^2}{2\lambda_{k+1}} \notag\\
	=&\frac{1}{2}\bigg(\frac{1}{\lambda_k}-\frac{1}{\lambda_{k+1}}(1-\sigma_k/\lambda_k)^2\bigg) |z_i^k|^2
	=\frac{\lambda_{k+1}-\lambda_k + 2\sigma_k  - \sigma_k^2/\lambda_k}{2\lambda_k^2 (\lambda_{k+1}/\lambda_k)} |z_i^k|^2,  \notag \\
    \leq  &\frac{\lambda_{k+1}-\lambda_k}{2}\frac{|z^k_i|^2}{\lambda_k^2}
    +\sigma_k \frac{|z^k_i|^2}{\lambda_k^2},
\end{align}
		where the first equality follows from \eqref{def-orig-psi}, 
        the second equality follows from \eqref{lemma-psi-4},
        and the last inequality follows from the monotonicity of $\{\lambda_k\}$.  
        Thus, we see that
		\eqref{lemma-psi-3-1} holds.
		
			\item $\lambda_{k+1} g_i({\bm x}^{k+1}) + z_i^{k+1}< 0$ and   $\lambda_k g_i({\bm x}^{k+1})+ z_i^k\geq 0$.
			
			In this case, we have that $g_i({\bm x}^{k+1})<0$. In addition, in view of \eqref{alg-3} and the fact that  $g_i({\bm x}^{k+1})\geq -z_i^k/\lambda_k$, it holds that
			\begin{align}\label{lemma-Psi-5}
				z_i^{k+1}=z_i^k+\sigma_k g_i({\bm x}^{k+1}).
				\end{align}
			  Hence, we have $0>\lambda_{k+1} g_i({\bm x}^{k+1}) + z_i^{k+1}\!=\!\lambda_{k+1} g_i({\bm x}^{k+1}) + z_i^k + \sigma_k g_i({\bm x}^{k+1})$, which further implies that
				\begin{align}\label{lemma-Psi-6}
			g_i({\bm x}^{k+1}) < -z_i^k/(\lambda_{k+1}+\sigma_k).
				\end{align}
			
			We then deduce from \eqref{def-orig-psi} that
			\begin{align}
				&\psi_{\lambda_{k+1}}(g_i({\bm x}^{k+1}),z_i^{k+1})-\psi_{\lambda_k}(g_i({\bm x}^{k+1}),z_i^k)  
				= - \frac{|z_i^{k+1}|^2}{2\lambda_{k+1}}-g_i({\bm x}^{k+1}) z_i^k-\frac{\lambda_k}{2} g_i({\bm x}^{k+1})^2	\notag	\\
				&= 	\sigma_k^{-1} |z_i^{k+1}-z_i^k|^2 - \frac{|z_i^{k+1}|^2}{2\lambda_{k+1}}	-	g_i({\bm x}^{k+1}) z_i^{k+1} - \frac{\lambda_k}{2} g_i({\bm x}^{k+1})^2 \notag \\
				&\leq  \sigma_k^{-1} |z_i^{k+1}-z_i^k|^2 - 
				\left( \frac{|z_i^{k+1}|^2}{2\lambda_{k+1}}	-	\frac{z_i^k z_i^{k+1}}{\lambda_k} + \frac{|z_i^k|^2}{2(\lambda_{k+1}+\sigma_k)^2/\lambda_k}\right) \notag \\
					&\leq  \sigma_k^{-1} |z_i^{k+1}-z_i^k|^2 - 
				\left( \frac{|z_i^{k+1}|^2}{2(\lambda_{k+1}+\sigma_k)^2/\lambda_k}	-	\frac{z_i^k z_i^{k+1}}{\lambda_k} + \frac{|z_i^k|^2}{2(\lambda_{k+1}+\sigma_k)^2/\lambda_k}\right) \notag \\
					&=  \left(\sigma_k^{-1}-\frac{\lambda_k}{2(\lambda_{k+1}+\sigma_k)^2}\right)|z_i^{k+1}-z_i^k|^2+\left(\frac{1}{\lambda_k}-\frac{\lambda_k}{(\lambda_{k+1}+\sigma_k)^2} \right) z_i^k z_i^{k+1} \notag \\
					&\leq    \sigma_k^{-1} |z_i^{k+1}-z_i^k|^2 + \frac{(\lambda_{k+1}+\lambda_k)(\lambda_{k+1}-\lambda_k)+2\lambda_{k+1}\sigma_k+\sigma_k^2}{\lambda_k(\lambda_{k+1}+\sigma_k)^2}  z_i^k z_i^{k+1} \notag \\
                    &\leq    \sigma_k^{-1} |z_i^{k+1}-z_i^k|^2 + 
                    \frac{2(\lambda_{k+1}-\lambda_k)}{\lambda_k^2}|z^k_i|^2 + \frac{2\sigma_k}{\lambda_k^2}|z^k_i|^2 + \frac{\sigma_k}{\lambda_k^2}|z^k_i|^2, \notag
			\end{align}
			where the second equality follows from \eqref{lemma-Psi-5},   the first inequality follows from \eqref{lemma-Psi-6} and the fact that $g_i({\bm x}^{k+1})\geq -z_i^k/\lambda_k$, 
            and the last inequality follows from the monotonicity of $\{\lambda_k\}$ and the facts that $z^{k+1}_i< z^k_i$ (thanks to \eqref{lemma-Psi-5} and $g_i({\bm x}^{k+1})<0$) and $\sigma_k \leq \lambda_k$ (see Step 1 of Algorithm~\ref{alg}). Then, from the above display and \eqref{lemma-psi-3}, we can verify \eqref{lemma-psi-3-1}.
	\end{enumerate}
	\end{proof}

The next lemma concerns the change in $\cL_\lambda$ after performing one iteration of Algorithm~\ref{alg}. We again consider both the general stepsize in $[0,1]$ and the short stepsize. This lemma will play an important role in the next section in our convergence analysis under the open-loop and the short stepsize schedules.

\begin{lemma}[One iteration progress]\label{lemma-od}
Suppose that Assumption~\ref{assu-function} holds.
	Let $\{({\bm x}^k,{\bm z}^k,{\bm v}^k)\}$ be generated by Algorithm \ref{alg}.
	 Let $\cL_\lambda$ be defined in \eqref{lagrangian}, $G_{{\bm z}, \lambda}({\bm x})$ be defined in \eqref{def-G}, and $L_{\Psi}$ be given in \eqref{lip-psi}. Then it holds that for all $k\ge 0$,
	\begin{align*}
		\cL_{\lambda_k}({\bm x}^{k+1},{\bm z}^k)\leq \cL_{\lambda_k}({\bm x}^k,{\bm z}^k) - \alpha_k G_{{\bm z}^k, \lambda_k}({\bm x}^k) + \frac{L_{\Psi}({\bm x}^k,{\bm z}^k)+L_f}{2}\alpha_k^2 \|{\bm v}^k-{\bm x}^k\|^2.
	\end{align*}
     Moreover, if $\alpha_k$ is the short stepsize in \eqref{short_stepsize_haha} with $({\bm x}^k,{\bm z}^k,{\bm v}^k,\lambda_k)$ in place of $({\bm x},{\bm z},{\bm v}^+,\lambda)$,
then
\begin{align*}
\mathcal{L}_{\lambda_k}({\bm x}^{k+1},{\bm z}^k)
\leq   \mathcal{L}_{\lambda_k}({\bm x}^k,{\bm z}^k) - 0.5 \alpha_k G_{{\bm z}^k, \lambda_k}({\bm x}^k).
	\end{align*}
\end{lemma}
\begin{proof}
	We have for any $k\ge 0$ that
	\begin{align}
		\mathcal{L}_{\lambda_k}({\bm x}^{k+1},{\bm z}^k) \overset{\rm (a)}{=}&  F_{\lambda_k}({\bm x}^k+\alpha_k({\bm v}^k- {\bm x}^k),{\bm z}^k)+h({\bm x}^k+\alpha_k({\bm v}^{k}-{\bm x}^k))\notag \\
		\overset{\rm (b)}{\leq} & F_{\lambda_k}({\bm x}^k,{\bm z}^k)+\alpha_k \langle \nabla_{\bm x} F_{\lambda_k}({\bm x}^k,{\bm z}^k), {\bm v}^k-{\bm x}^k \rangle\notag\\
        &\ \ \ \ + \frac{L_{\Psi}({\bm x}^k,{\bm z}^k)+L_f}{2}\alpha_k^2\|{\bm v}^k- {\bm x}^k\|^2 +(1-\alpha_k) h({\bm x}^k)+\alpha_k h({\bm v}^k) \notag  \\
		\overset{\rm (c)}{=} &  \mathcal{L}_{\lambda_k}({\bm x}^k,{\bm z}^k) -\alpha_k G_{{\bm z}^k, \lambda_k}({\bm x})+\frac{L_{\Psi}({\bm x}^k,{\bm z}^k)+L_f}{2}\alpha_k^2\|{\bm v}^k- {\bm x}^k\|^2, \label{lemma-od-4}
	\end{align}
	where (a) follows from \eqref{lagrangian-F},
	(b) follows from the convexity of $h$ and the fact that $F_{\lambda_k}(\cdot,{\bm z}^k)=f(\cdot)+\Psi_{\lambda_k}(\cdot,{\bm z}^k)$ is Lipschitz differentiable with modulus $L_{\Psi}({\bm x}^k,{\bm z}^k)+L_f$ thanks to Lemma \ref{lemma-pre}(iv),
    and (c) follows from \eqref{lagrangian-F} and the definition of $G_{{\bm z}, \lambda}({\bm x})$ in \eqref{def-G}.

    Finally, suppose that $\alpha_k$ is the short stepsize. When ${G_{{\bm z}^k,\lambda_k}({\bm x}^k)}=0$, the claimed inequality clearly holds as $\alpha_k = 0$. When ${G_{{\bm z}^k,\lambda_k}({\bm x}^k)}>0$, we have
    \begin{align}
		\mathcal{L}_{\lambda_k}({\bm x}^{k+1},{\bm z}^k)
			&\leq \mathcal{L}_{\lambda_k}({\bm x}^k,{\bm z}^k)-
			\min\left\{1, \frac{G_{{\bm z}^k, \lambda_k}({\bm x}^k)}{(L_\Psi({\bm x}^k,{\bm z}^k)+L_f)\|{\bm v}^k-{\bm x}^k\|^2}\right\} G_{{\bm z}^k, \lambda_k}({\bm x}^k) \notag\\
			&\ \ \ \ + \frac{L_{\Psi}({\bm x}^k,{\bm z}^k)+L_f}{2}\min\left\{1, \frac{G_{{\bm z}^k, \lambda_k}({\bm x}^k)}{(L_\Psi({\bm x}^k,{\bm z}^k)+L_f)\|{\bm v}^k - {\bm x}^k\|^2}\right\}^2\|{\bm v}^k - {\bm x}^k\|^2 \notag \\
&\le \mathcal{L}_{\lambda_k}({\bm x}^k,{\bm z}^k) - \frac{G_{{\bm z}^k, \lambda_k}({\bm x}^k)}{2}\min\left\{1, \frac{G_{{\bm z}^k, \lambda_k}({\bm x}^k)}{(L_{\Psi}({\bm x}^k,{\bm z}^k)+L_f)\| {\bm x}^k - {\bm v}^k\|^2} \right\},  \notag
	\end{align}
	where the first inequality follows upon substituting the short stepsize in \eqref{lemma-od-4}.
\end{proof}

We next present our general convergence results for Algorithm~\ref{alg}, which are the basis for our convergence analysis under specific stepsize schedules in Section~\ref{sec4} and additional structural assumptions on $h$ in Section~\ref{sec-con-ss}. We first define another {\em nonnegative} auxiliary quantity that is related to optimality. Let
\begin{align}\label{def-T}
 T_{\lambda}({\bm x},{\bm z}):=[\mathcal{L}_{\lambda}({\bm x},{\bm z})-L^*]_+,
\end{align}
where $\cL_\lambda$ is defined in \eqref{lagrangian} and $L^*$ is the optimal value of \eqref{problem}.
For $\{({\bm x}^k,{\bm z}^k)\}$ and $\{\lambda_k\}$ from Algorithm \ref{alg}, we define
\begin{align}\label{def-Tk}
	T_{k} :=T_{\lambda_k}({\bm x}^k,{\bm z}^k)=[\mathcal{L}_{\lambda_k}({\bm x}^k,{\bm z}^k)-L^*]_+\ \ \ \ \ \forall k\ge 0.
\end{align}
In the next theorem, we bound the feasibility violation and the primal suboptimality in terms of $\{T_k\}$ and $\{\lambda_k\}$.

\begin{theorem}[Feasibility and primal suboptimality]\label{lemma-fea-fun-bound}
Suppose that Assumption~\ref{assu-function} holds.
	Let $\{({\bm x}^k,{\bm z}^k)\}$ be generated by Algorithm \ref{alg} and $T_k$ be given in \eqref{def-Tk}.
    Then it holds that
    \begin{align*}
      &\| [{\bm g}({\bm x}^k)]_+\| = \mathcal{O}(\max\{(T_k/\lambda_k)^\frac12,1/\lambda_k\})= \mathcal{O}(\max\{T_k,1/\lambda_k\}),\\
      &|f({\bm x}^k)+h({\bm x}^k)-L^*|= \mathcal{O}(\max\{T_k,1/\lambda_k\}),
    \end{align*}
    where $L^*$ is the optimal value of problem \eqref{problem}.
\end{theorem}
\begin{proof}
Notice that for all $k\ge 0$,
    \begin{align}
		&f({\bm x}^k)+ h({\bm x}^k)-L^* + \frac{\lambda_k}{2}\|[{\bm g}({\bm x}^k)]_+\|^2 \notag \\
        \leq & f({\bm x}^k)+ h({\bm x}^k)-L^* + \frac{\lambda_k}{2}\sum_{i=1}^m\bigg(\bigg[g_i({\bm x}^k)+\frac{z^k_i}{\lambda_k}\bigg]_+^2 -\frac{(z^k_i)^2}{\lambda_k^2} + \frac{(z^k_i)^2}{\lambda_k^2}\bigg)\notag\\
        \leq & T_k  + \|{\bm z}^k\|^2/(2\lambda_k), \label{lemma-fea-fun-bound-1}
		\end{align}
        where the first inequality holds because ${\bm z}^k \ge {\bm 0}$ (see Remark~\ref{remark-bound}), and the second inequality follows from the definition of $T_k$ in \eqref{def-Tk}.
        
		Recall that $({\bm x}^*,{\bm z}^*)$ is a KKT point of problem \eqref{problem} from Assumption \ref{assu-function}(iv). 
		From \cite[Corollary 3.4]{Rockafellar-1973}, we see that $({\bm x}^*,{\bm z}^*)$ is also a saddle point of the function $({\bm x},{\bm z})\mapsto  f({\bm x})+h({\bm x}) + \langle {\bm z}, {\bm g}({\bm x})\rangle-{\rm Ind}_{\R_+^m}({\bm z})$, which implies that for all $k\ge 0$,
		\begin{align}
		L^*=&f({\bm x}^*)+h({\bm x}^*)
		 \leq f({\bm x}^k)+h({\bm x}^k) +\langle {\bm z}^*, g({\bm x}^k)\rangle  \notag\\
         \leq& f({\bm x}^k)+h({\bm x}^k) +\langle {\bm z}^*, [g({\bm x})^k]_+)\rangle 
		 \leq f({\bm x}^k)+h({\bm x}^k)+\|{\bm z}^*\| \| [{\bm g}({\bm x}^k)]_+\|, \label{lemma-fea-fun-bound-2}  
		\end{align}
		where the second inequality holds because ${\bm z}^*\geq {\bm 0}$,  and the last inequality follows from the Cauchy-Schwarz inequality.

    From the last two displays, we have for all $k\ge 0$,
    \begin{align*}
    &-\|{\bm z}^*\| \| [{\bm g}({\bm x}^k)]_+\|+ \frac{\lambda_k}{2}\|[{\bm g}({\bm x}^k)]_+\|^2  \\
    \leq & (f({\bm x}^k)+ h({\bm x}^k)-L^* +\|{\bm z}^*\| \| [{\bm g}({\bm x}^k)]_+\|) -\|{\bm z}^*\| \| [{\bm g}({\bm x}^k)]_+\|+ \frac{\lambda_k}{2}\|[{\bm g}({\bm x}^k)]_+\|^2 \\
    \leq & T_k + \|{\bm z}^k\|^2/(2\lambda_k).
    \end{align*}
    Then, from the first and third rows of the above display, we have that 
     \begin{align*}
     \| [{\bm g}({\bm x}^k)]_+\| \leq & \bigg(\|{\bm z}^*\|+ \sqrt{\|{\bm z}^*\|^2+2\lambda_k(T_k+\|{\bm z}^k\|^2/(2\lambda_k))}\bigg)/\lambda_k.  
     \end{align*}
     Moreover, it follows from \eqref{lemma-fea-fun-bound-1} and \eqref{lemma-fea-fun-bound-2} that 
     \[
     |f({\bm x}^k)+h({\bm x}^k)-L^*|\leq \max\bigg\{T_k+\frac{\|{\bm z}^k\|^2}{2\lambda_k}, \,\|{\bm z}^*\| \|[{\bm g}({\bm x}^k)]_+\|\bigg\}.
     \]
  The desired assertion now follows immediately from the last two displays, the boundedness of $\{{\bm z}^k\}$ (see \eqref{remark-bound-zk}), and the observation that $2\sqrt{T_k/\lambda_k}\le T_k + 1/\lambda_k$.
\end{proof}

In view of Theorem~\ref{lemma-fea-fun-bound}, we see that the key to establishing the convergence of Algorithm~\ref{alg} is to derive conditions to obtain {\em vanishing} bounds on the {\em nonnegative} auxiliary sequences $\{T_k\}$ (by, e.g., suitable assumptions on $\{\alpha_k\}$, etc).
Specifically, we have the following convergence result based on the additional assumptions \eqref{general_conditions} and \eqref{the-con-1-3}: we will show in Section~\ref{sec4} that \eqref{general_conditions} and \eqref{the-con-1-3} can be guaranteed by properly choosing $\{\alpha_k\}$, $\{\sigma_k\}$ and $\{\lambda_k\}$.

\begin{theorem}[Global convergence]\label{the-con-0}
Suppose that Assumption~\ref{assu-function} holds.
	Let $\{({\bm x}^k,{\bm z}^k)\}$ be generated by Algorithm \ref{alg} and $L^*$ be the optimal value of problem \eqref{problem}. Let $\cL_\lambda$ be defined in \eqref{lagrangian}, $G_{{\bm z}, \lambda}({\bm x})$ be defined in \eqref{def-G}, and $T_k$ be defined in \eqref{def-Tk}. Assume that
    \begin{equation}\label{general_conditions}
\lim_{k\to\infty}T_k = 0.
   \end{equation}
	Then it holds that
    \begin{equation}\label{the-con-1-2-0}
    \lim_{k\to \infty}\|[{\bm g}({\bm x}^k)]_+\| = \lim_{k\to\infty} |f({\bm x}^k)+h({\bm x}^k)-L^*| = \lim_{k\rightarrow \infty} |\cL_{\lambda_k}({\bm x}^k,{\bm z}^k)-L^*|=0.
\end{equation}
    Moreover, if we assume further that problem \eqref{problem} has a Slater point $\tilde{\bm x}$, i.e., ${\bm g}(\tilde{\bm x})< {\bm 0}$ with $\tilde{{\bm x}}\in {\rm dom}\, h$, and that there exist positive sequence $\{\xi_k\}$ and $\iota > 0$ such that
      \begin{equation}\label{the-con-1-3}\textstyle
   \lim_{k\rightarrow \infty} \sum_{i=0}^k \frac{\xi_i}{\sum_{j=0}^k \xi_j} (G_{{\bm z}^i,\lambda_i}({\bm x}^i))^\iota= 0,
    \end{equation}
    then there exists a subsequence of $\{({\bm x}^k,[\lambda_k {\bm g}({\bm x}^k)+ {\bm z}^k]_+)\}$ such that
    all of its cluster points are KKT points of \eqref{problem}.
\end{theorem}
\begin{proof}
We deduce $\lim_{k\to \infty}\|[{\bm g}({\bm x}^k)]_+\| = \lim_{k\to\infty} |f({\bm x}^k)+h({\bm x}^k)-L^*|=0$ from \eqref{general_conditions} and Theorem~\ref{lemma-fea-fun-bound}.
Next, notice from the definition of $\mathcal{L}_\lambda$ in \eqref{lagrangian} that for all $k\ge 0$,
\begin{align*}
|\cL_{\lambda_{k}}({\bm x}^{k},{\bm z}^{k})- L^*| &= [\cL_{\lambda_{k}}({\bm x}^{k},{\bm z}^{k})- L^*]_+ + [L^* - \cL_{\lambda_{k}}({\bm x}^{k},{\bm z}^{k})]_+\\
&\le T_k + [L^* - (f({\bm x}^k)+h({\bm x}^k))]_+ + [-\Psi_{\lambda_k}({\bm x}^k,{\bm z}^k)]_+.
\end{align*}
Since $\|{\bm z}^k\|^2/\lambda_k\rightarrow 0$ (thanks to \eqref{remark-bound-zk} and $\lambda_k\rightarrow \infty$),  we have upon recalling the definition of $\Psi_{\lambda}$ in \eqref{def-psi} that $\liminf_{k\rightarrow \infty}\Psi_{\lambda_k}({\bm x}^k,{\bm z}^k) \geq \lim_{k\rightarrow \infty} - \|{\bm z}^k\|^2/(2\lambda_k)=0$. Using this together with the above display and the facts that $T_k \to 0$ (see \eqref{general_conditions}) and $f({\bm x}^k)+h({\bm x}^k)-L^*\rightarrow 0$, we deduce that $|\cL_{\lambda_{k}}({\bm x}^{k},{\bm z}^{k})- L^*|\rightarrow 0$.

    Finally, we further assume the existence of the Slater point $\tilde{\bm x}$ and \eqref{the-con-1-3} to prove the last assertion. We first prove the following implication:
    \begin{align}\label{the-con-9}
    \textstyle
    {\bm 0}\in \sum_{i=1}^m z_i \nabla g_i({\bm x})+\mathcal{N}_{{\rm dom}\, h}({\bm x}), \, \, {\bm z}^T {\bm g}({\bm x})=0, \,\, {\bm z}\geq {\bm 0} \,\, \Longrightarrow\,\, {\bm z}={\bm 0}.
    \end{align}
    Indeed, for $({\bm x},{\bm z})$ satisfying the left-hand side of \eqref{the-con-9}, we have
    \[
    0\overset{\rm (a)}\leq \bigg\langle \sum_{i=1}^m z_i \nabla g_i({\bm x}), \tilde{\bm x} -{\bm x}\bigg\rangle\overset{\rm (b)}\leq \sum_{i=1}^m z_i(g_i(\tilde{\bm x})-g_i({\bm x}))= {\bm z}^T {\bm g}(\tilde{\bm x})\leq 0,
    \]
    where (a) follows from the first inclusion in \eqref{the-con-9} and the definition of the normal cone, (b) follows from the convexity of $g_i$, the  equality holds because ${\bm z}^T {\bm g}({\bm x})=0$,  and the last inequality holds because ${\bm z}\geq {\bm 0}$ and ${\bm g}(\tilde{\bm x})<{\bm 0}$.
    The above display gives ${\bm z}^T{\bm g}(\tilde{\bm x})=0$, which together with ${\bm z}\geq 0$ and ${\bm g}(\tilde{\bm x})<0$ implies ${\bm z}={\bm 0}$.

    From \eqref{the-con-1-3}, it is clear that $\liminf_{k\rightarrow \infty}G_{{\bm z}^k,\lambda_k}({\bm x}^k)=0$. 
    Let $\mathcal{\tilde{K}}$ be an index set such that $\{{\bm x}^k\}_{k\in \mathcal{\tilde{K}}}$ is a subsequence satisfying $\lim_{k\rightarrow \infty, k\in \mathcal{\tilde{K}}}G_{{\bm z}^k,\lambda_k}({\bm x}^k)=0$.
    Define 
    \[
    \hat\epsilon_k:= G_{{\bm z}^k,\lambda_k}({\bm x}^k)=\langle \nabla_{\bm x} F_{\lambda_k}({\bm x}^k,{\bm z}^k), {\bm x}^k - {\bm v}^k\rangle + h({\bm x}^k) - h({\bm v}^k)\ \ \ \forall k \ge 0.
    \]
    Then we have $\lim_{k\rightarrow \infty, k\in \mathcal{\tilde{K}}}\hat\epsilon_k=0$. From \eqref{alg-1} and the definition of $\hat{\epsilon}_k$, we have that
    \begin{align*}
     \langle  \nabla_{\bm x} F_{\lambda_k}({\bm x}^k,{\bm z}^k),{\bm x}^k\rangle + h({\bm x}^k) = & \langle  \nabla_{\bm x} F_{\lambda_k}({\bm x}^k,{\bm z}^k),{\bm v}^k\rangle + h({\bm v}^k)+\hat{\epsilon}_k \\
     \leq & \langle  \nabla_{\bm x} F_{\lambda_k}({\bm x}^k,{\bm z}^k),{\bm x}\rangle + h({\bm x})+\hat{\epsilon}_k \quad \forall \, {\bm x}\in {\rm dom}\,h,
    \end{align*}
    which further implies that $-\nabla_{\bm x} F_{\lambda_k}({\bm x}^k,{\bm z}^k)\in \partial_{\hat{\epsilon}_k}h({\bm x}^k)$. From this fact and the definition of $F_{\lambda_k}$ (see \eqref{lagrangian-F}), we have that
    \begin{align}\label{the-con-10}
    \textstyle
   {\bm 0}\in \nabla f({\bm x}^k)+\sum_{i=1}^m[\lambda_k g_i({\bm x}^k)+{\bm z}^k_i]_+\nabla g_i({\bm x}^k)+ \partial_{\hat{\epsilon}_k}h({\bm x}^k)\ \ \ \forall k\ge 0.
    \end{align}
    
    We claim that $\{[\lambda_k {\bm g}({\bm x}^k)+{\bm z}^k]_+\}_{k\in \tilde{\cal K}}$ is bounded. Suppose to the contrary that $\{[\lambda_k {\bm g}({\bm x}^k)+{\bm z}^k]_+\}_{k\in \tilde{\cal K}}$ is unbounded.
   By passing to a further subsequence (we denote its index set by $\hat {\cal K}$) if necessary, we assume that $\lim_{k\to\infty,k\in \hat{\cal K}}[\lambda_k {\bm g}({\bm x}^k)+{\bm z}^k]_+ = \infty$ and $\{({\bm x}^k,[\lambda_k {\bm g}({\bm x}^k) +{\bm z}^k]_+/\|[\lambda_k {\bm g}({\bm x}^k) +{\bm z}^k]_+\|)\}_{k\in \hat {\cal K}}$ converges to some $(\bar{\bm x},\bar{\bm z})$ satisfying ${\bm g}(\bar{\bm x})\leq {\bm 0}$ (see \eqref{the-con-1-2-0}) and $\bar{\bm z}\geq {\bm 0}$ with $\bar{\bm z}\neq {\bm 0}$.    
    Dividing both sides of \eqref{the-con-10} by $\|[\lambda_k {\bm g}({\bm x}^k) +{\bm z}^k]_+\|$ and passing to the limit along $k\in \hat {\cal K}$, we obtain
    \begin{align}\label{the-con-11}
    \textstyle
      {\bm 0}\in \sum_{i=1}^m\bar{z}_i\nabla g_i(\bar{\bm x})+\mathcal{N}_{{\rm dom}\, h}(\bar{\bm x}), 
    \end{align}
    where we used the boundedness of $\{\nabla f({\bm x}^k)\}$ and the definitions of $\epsilon$-subdifferential and normal cone.\footnote{Precisely, if $\{t_k\}_{k\in \hat{\cal K}}$ is a positive sequence converging to zero and $\{{\bm v}^k\}_{k\in {\hat {\cal K}}}$ converges to some $\bar {\bm v}$ and satisfies $t_k^{-1}{\bm v}^k\in \partial_{\hat\epsilon_k}h({\bm x}^k)$ for all $k\in {\hat{\cal K}}$, then $\langle {\bm v}^k, {\bm x} - {\bm x}^k\rangle\le t_k(h({\bm x}) - h({\bm x}^k) + \hat\epsilon_k)$ for all ${\bm x}\in {\rm dom}\, h$, which implies $\langle \bar {\bm v},{\bm x} - \bar {\bm x}\rangle \le 0$ for all ${\bm x}\in {\rm dom}\, h$. We obtain \eqref{the-con-11} upon applying the above observation to $t_k := \|[\lambda_k {\bm g}({\bm x}^k) +{\bm z}^k]_+\|^{-1}$ and ${\bm v}^k := -t_k(\nabla f({\bm x}^k)+\sum_{i=1}^m[\lambda_k g_i({\bm x}^k)+{\bm z}^k_i]_+\nabla g_i({\bm x}^k))$, and noting that $\bar {\bm x}\in {\rm dom}\,h$ since $f(x^k) + h(x^k) \to L^*$ (see \eqref{the-con-1-2-0}) and $h$ is closed.} Next, if $g_j(\bar{\bm x})<0$ for some $j\in [m]$, then, for all sufficiently large $k\in \hat{\cal K}$, we have 
    \begin{align*}
    \frac{[\lambda_k g_j({\bm x}^k)+z^k_j]_+}{\|[\lambda_k {\bm g}({\bm x}^k) +{\bm z}^k]_+\|}=\frac{\lambda_k [g_j({\bm x}^k)+z^k_j/\lambda_k]_+}{\|[\lambda_k {\bm g}({\bm x}^k) +{\bm z}^k]_+\|}= 0,
    \end{align*}
    where the last equality holds because $z^k_j/\lambda_k\rightarrow 0$ (thanks to \eqref{remark-bound-zk} and $\lambda_k\rightarrow \infty$) and $g_j({\bm x}^k)\rightarrow g_j(\bar{\bm x}) < 0$.
    Hence, we have $\bar{z}_j=0$ if $g_j(\bar{\bm x})<0$, which further implies that $\bar{\bm z}^T{\bm g}(\bar{\bm x})=0$.
     This fact together with \eqref{the-con-11} and the fact that $\bar {\bm z}\neq {\bm 0}$ contradicts the implication \eqref{the-con-9}. 
    Hence, we conclude that $\{[\lambda_k {\bm g}({\bm x}^k)+{\bm z}^k]_+\}_{k\in \tilde{\cal K}}$ is bounded.

    Since $\{({\bm x}^k,[\lambda_k {\bm g}({\bm x}^k)+{\bm z}^k]_+)\}_{k\in \tilde{\cal K}}$ is bounded, we can (arbitrarily) choose a further convergent subsequence with index set ${\cal\tilde{K}}_1\subseteq {\cal \tilde{K}}$ such that $\lim_{k\rightarrow \infty, k\in \mathcal{\tilde{K}}_1}({\bm x}^k,[\lambda_k {\bm g}({\bm x}^k)+{\bm z}^k]_+) = (\hat{\bm x},\hat{\bm z})$ for some $\hat{\bm x}$ and $\hat{\bm z}$. Then $g(\hat{{\bm x}})\leq {\bm 0}$ (see \eqref{the-con-1-2-0})
    and $\hat{\bm z}\geq 0$.
    Notice that if $\hat{z}_i> 0$, then $g_i(\hat{\bm x})=0$, for otherwise we have
    a contradiction as follows: 
    \[
    0<\hat{z}_i=\lim_{k\rightarrow \infty, k\in \tilde{\mathcal{K}}_1}[\lambda_k g_i({\bm x}^k)+z^k_i]_+= \lim_{k\rightarrow \infty, k\in \tilde{\mathcal{K}}_1}\lambda_k[g_i({\bm x}^k)+z^k_i/\lambda_k]_+=0,\]
   where the last equality holds because $z^k_i/\lambda_k\rightarrow 0$ (thanks to \eqref{remark-bound-zk} and $\lambda_k\rightarrow \infty$) and $g_i({\bm x}^k)\rightarrow g_i(\hat{\bm x})<0$.
   In addition, invoking \cite[Theorem 2.4.2(ix)]{Zalinescu-2002} and passing to the limit along $k\in \mathcal{\tilde{K}}_1$ in \eqref{the-con-10}, we have  
 \[\textstyle
   {\bm 0}\in \nabla f(\hat{\bm x})+\sum_{i=1}^m \hat{z}_i \nabla g_i(\hat{\bm x})+ \partial h(\hat{\bm x}).
  \]
  Therefore, we have verified that $(\hat{\bm x},\hat{\bm z})$ satisfies the KKT condition (see also \eqref{kkt}). The desired conclusion follows from this and the arbitrariness of $\mathcal{\tilde{K}}_1$.
\end{proof}

\begin{remark}[Constructability of subsequence index set $\tilde{\cal K}$]
The proof of Theorem \ref{the-con-0} reveals that any convergent subsequence of $\{({\bm x}^k,[\lambda_k {\bm g}({\bm x}^k)+{\bm z}^k]_+)\}_{\tilde{\cal K}}$ converges to a KKT point, where $\tilde{\cal K}$ is such that $\lim_{k\rightarrow \infty, k\in \mathcal{\tilde{K}}}G_{{\bm z}^k,\lambda_k}({\bm x}^k)=0$. Here, we discuss how such an index set can be identified based on \eqref{the-con-1-3}.

Let $\mathcal{G}_k :=\sum_{i=0}^k \frac{\xi_i}{\Gamma_k} (G_{{\bm z}^i,\lambda_i}({\bm x}^i))^\iota$, where $\Gamma_k:=\sum_{i=0}^k \xi_i$.
Then $\mathcal{G}_k\ge 0$ for all $k$ and $\mathcal{G}_k \to 0$. Hence, there exists a subsequence satisfying
$\mathcal{G}_{i_k} \leq \mathcal{G}_{i_{k}-1}$ for all $k$. Moreover, for any subsequence $\{\mathcal{G}_{i_k}\}$ satisfying $\mathcal{G}_{i_k} \leq \mathcal{G}_{i_{k}-1}$ for all $k$, we have
\[
(G_{{\bm z}^{i_k},\lambda_{i_k}}({\bm x}^{i_k}))^\iota=(\Gamma_{i_k} \mathcal{G}_{i_k}-\Gamma_{i_k-1} \mathcal{G}_{i_k-1})/\xi_{i_k}
=(\Gamma_{i_k-1} (\mathcal{G}_{i_k}-\mathcal{G}_{i_k-1}))/\xi_{i_k}+ \mathcal{G}_{i_k} \leq \mathcal{G}_{i_k}.
\]
Therefore, if we define 
\[\textstyle
\tilde{\mathcal{K}}:=\big\{ k\in \mathbb{N}_+  : \, \sum_{i=0}^k \frac{\xi_i}{\Gamma_k}(G_{{\bm z}^i, \lambda_i}({\bm x}^i))^\iota\leq \sum_{i=0}^{k-1} \frac{\xi_i}{\Gamma_{k-1}} (G_{{\bm z}^i, \lambda_i}({\bm x}^i))^\iota\big\},
\]
then $\lim_{k\rightarrow \infty, k\in \tilde{\mathcal{K}}} G_{{\bm z}^k, \lambda_k}({\bm x}^k)=0$ and each accumulation point of $\{({\bm x}^k,[\lambda_k {\bm g}({\bm x}^k)+{\bm z}^k]_+)\}_{\tilde{\mathcal{K}}}$ is a KKT point.
\end{remark}

\section{Convergence analysis of Algorithm~\ref{alg}}\label{sec4}

In view of Theorem~\ref{the-con-0}, the key to establishing convergence of Algorithm~\ref{alg} is to guarantee \eqref{general_conditions} and \eqref{the-con-1-3}. We show that these conditions can be achieved if we suitably choose the parameters $\{\alpha_k\}$, $\{\lambda_k\}$ and $\{\sigma_k\}$ in Algorithm~\ref{alg}.

\subsection{Open-loop stepsize}

In this section, we consider the following assumption for $\{\alpha_k\}$, $\{\lambda_k\}$ and $\{\sigma_k\}$ in Algorithm \ref{alg}.

\begin{assumption}[Assumptions for  parameters in Algorithm \ref{alg}]\label{assu-parameter}
	\begin{enumerate}[{\rm (i)}]		
    \item The positive penalty parameter $\{\lambda_k\}$ is increasing, $\lambda_k=\Theta(k^{\tau})$ and $\lambda_{k+1} - \lambda_k = \mathcal{O}(k^{-(1-\tau)})$ for some $\tau \in (0,1)$.
    \item The stepsize $\{\alpha_k\}\subset(0,1]$ is nonincreasing and  $\alpha_k=\Theta(k^{-p})$, where $p\in (\tau,1)$.	
     \item The positive sequence  $\{\sigma_k\}$ satisfies  $\sigma_k \leq \lambda_k$ and $\sigma_k=\Theta(k^{-(1+\gamma)})$ for some $\gamma>0$.
	\end{enumerate}
\end{assumption}

In order to invoke Theorems~\ref{lemma-fea-fun-bound} and~\ref{the-con-0} to deduce global convergence and complexity, it suffices to study the convergence behaviors of $\{T_k\}$ and $\{G_{{\bm z}^k,\lambda_k}({\bm x}^k)\}$. In the next proposition, we first show that $\{T_k\}$ satisfies a certain difference inequality and then establish the convergence rate of $\{T_k\}$ and a weighted average of $\{G_{{\bm z}^k,\lambda_k}({\bm x}^k)\}$.  

\begin{proposition}\label{prop-rate-T}
Suppose that Assumptions~\ref{assu-function} and \ref{assu-parameter} hold.
	Let $\{({\bm x}^k,{\bm z}^k,{\bm v}^k)\}$ be generated by Algorithm \ref{alg} and $L^*$ be the optimal value of problem \eqref{problem}. Let $\cL_\lambda$ be defined in \eqref{lagrangian}, $G_{{\bm z}, \lambda}({\bm x})$ be defined in \eqref{def-G}, $L_{\Psi}$ be given in \eqref{lip-psi}, and $T_k$ be given in \eqref{def-Tk}. 
     Then, there exist some positive constants $c_1$, $c_2$ and $c_3$ such that for all $k\ge 0$,
		\begin{align}
       & \cL_{\lambda_{k+1}}({\bm x}^{k+1},{\bm z}^{k+1})- L^* - c_1 \max\bigg\{\frac{1}{\lambda_k}, \alpha^2_k\lambda_k,T_k\bigg\} \frac{\sigma_k + \lambda_{k+1}-\lambda_k}{\lambda_k} \notag\\
  & \quad \quad \, \leq  \cL_{\lambda_k}({\bm x}^k,{\bm z}^k) -L^*  - \alpha_k G_{{\bm z}^k,\lambda_k}({\bm x^k}) + \frac{L_{\Psi}({\bm x}^k,{\bm z}^k)+L_f}{2}\alpha_k^2 \|{\bm v}^k-{\bm x}^k\|^2, \label{prop-rate-T-1} \\
	&	T_{k+1}
		\leq  (1-\alpha_k) T_k
             + \frac{\lambda_k + c_2}{2}c_3\alpha_k^2 +c_1 \max\bigg\{\frac{1}{\lambda_k}, \alpha^2_k\lambda_k,T_k\bigg\} \frac{\sigma_k + \lambda_{k+1}-\lambda_k}{\lambda_k}.  \label{prop-rate-T-1-1}
		\end{align}
    Moreover, it holds that 
	\begin{align}
    T_k &= \textstyle \mathcal{O}\big(\max\big\{\frac{1}{k^{p-\tau}},\frac{1}{k^{1-p+\tau}}\big\}\big), \label{prop-rate-T-2}\\
   \textstyle\sum_{i=0}^k  \frac{\alpha_i}{\Gamma_k} G_{{\bm z}^i, \lambda_i}({\bm x}^{i}) &= \begin{cases}
   \mathcal{O}(1/k^{1-p}) & {\rm if}\ 2p-\tau > 1,\\
   \mathcal{O}(\log(k)/k^{1-p}) & {\rm if}\ 2p-\tau = 1,\\
   \mathcal{O}(1/k^{p-\tau}) & {\rm if}\ 2p-\tau < 1,
   \end{cases}\label{prop-rate-TG-2}
	\end{align}
	where $\Gamma_k := \sum_{i=0}^k \alpha_i$,
     and $(p,\tau,\gamma)$ are specified in Assumption~\ref{assu-parameter}.
\end{proposition}
\begin{proof}
 Applying Lemma \ref{lemma-od} with $({\bm x}^{k+1},{\bm x}^k,{\bm z}^k,{\bm v}^k,\alpha_k,\lambda_{k})$ in place of the sextuple $({\bm x}^+,{\bm x},{\bm z},{\bm v}^+,\alpha,\lambda)$, we have for all $k\ge 0$ that
            	\begin{align*}
		\cL_{\lambda_k}({\bm x}^{k+1},{\bm z}^k)\leq \cL_{\lambda_k}({\bm x}^k,{\bm z}^k) - \alpha_k G_{{\bm z}^k,\lambda_k}({\bm x^k}) + \frac{L_{\Psi}({\bm x}^k,{\bm z}^k)+L_f}{2}\alpha_k^2 \|{\bm v}^k-{\bm x}^k\|^2.
	\end{align*}
    From the last display, we see that 
    \begin{align}
   &\cL_{\lambda_{k+1}}({\bm x}^{k+1},{\bm z}^{k+1})- L^*-(\cL_{\lambda_{k+1}}({\bm x}^{k+1},{\bm z}^{k+1})-\cL_{\lambda_k}({\bm x}^{k+1},{\bm z}^k)) \notag\\
   \leq & \cL_{\lambda_k}({\bm x}^k,{\bm z}^k) -L^* - \alpha_k G_{{\bm z}^k,\lambda_k}({\bm x^k}) + \frac{L_{\Psi}({\bm x}^k,{\bm z}^k)+L_f}{2}\alpha_k^2 \|{\bm v}^k-{\bm x}^k\|^2.
   \label{prop-rate-T-3} 
    \end{align}
  Also, from \eqref{lemma-psi-1} and the bound on $\|[{\bm g}({\bm x}^k)]_+\|$ in Theorem~\ref{lemma-fea-fun-bound}, there exists $c_1 >0$ such that 
    \begin{align}\label{prop-rate-T-3-1}
    \cL_{\lambda_{k+1}}({\bm x}^{k+1},{\bm z}^{k+1})\!-\!\cL_{\lambda_k}({\bm x}^{k+1},{\bm z}^k)
    \!\leq\! c_1 \max\bigg\{\frac{1}{\lambda_k}, \alpha^2_k\lambda_k, T_k\bigg\} \frac{\sigma_k \!+\! \lambda_{k+1}\!-\!\lambda_k}{\lambda_k}\ \ \forall k\ge 0.
    \end{align}
    Then, the inequality \eqref{prop-rate-T-1} follows from the \eqref{prop-rate-T-3} and \eqref{prop-rate-T-3-1}.
    In addition, we have from \eqref{prop-rate-T-1} that for all $k\ge 0$,
     \begin{align}
   T_{k+1}  
   \leq & (1 - \alpha_k) T_k \!+ \! \frac{L_{\Psi}({\bm x}^k,{\bm z}^k)\!+\!L_f}{2}\alpha_k^2 \|{\bm v}^k\!-\!{\bm x}^k\|^2 
  \! +\! c_1 \max\bigg\{\frac{1}{\lambda_k}, \alpha^2_k\lambda_k,T_k\bigg\} \frac{\sigma_k \!+\! \lambda_{k+1} \!-\! \lambda_k}{\lambda_k}\notag \\
      &\!\leq\! (1 - \alpha_k)T_k + 
   \underbrace{\frac{\lambda_k + c_2}{2}c_3\alpha_k^2 +c_1 \max\bigg\{\frac{1}{\lambda_k}, \alpha^2_k\lambda_k, T_k\bigg\} \frac{\sigma_k + \lambda_{k+1}-\lambda_k}{\lambda_k}}_{=:\beta_k}, \label{prop-rate-T-4}
    \end{align}
    where the first inequality follows from the definition of $T_k$ in \eqref{def-Tk} and Lemma~\ref{lemma-G-T},
    and the second inequality holds for some positive constants $c_2$ and $c_3$
    since we have $\|{\bm v}^k-{\bm x}^k\|\leq D$ for any $k$ by Assumption~\ref{assu-function}(ii) and $L_{\Psi}({\bm x}^k,{\bm z}^k)=\mathcal{O}(\lambda_k)$ thanks to \eqref{Lip-Lpsi}.
    This establishes the claimed difference inequality for $T_k$ in \eqref{prop-rate-T-1-1}. 

   Under Assumption~\ref{assu-parameter}, there exist $c_4$, $c_5>0$ and positive integer $k_0$ such that for all $k\geq k_0$
   \begin{align}
   &\frac{\lambda_k + c_2}{2}c_3\alpha_k^2 \leq \frac{c_4}{(k+1)^{2p-\tau}},\ \
   c_1\frac{\sigma_k + \lambda_{k+1}-\lambda_k}{\lambda_k}\leq \frac{c_4}{k+1},
   \label{prop-rate-T-5}\\ 
   &-\alpha_k +\frac{c_4}{k+1}\leq -0.5\alpha_k, \ \ \frac{c_4}{\lambda_k}\le \frac{c_5}{(k+1)^\tau},\ \ \frac{\alpha_k^2\lambda_kc_4}{k+1}\le \frac{c_5}{(k+1)^{2p-\tau}}.\label{prop-rate-T-522}
   \end{align}
    
    From \eqref{prop-rate-T-4}, \eqref{prop-rate-T-5} and \eqref{prop-rate-T-522}, we have for any $k\geq k_0$ that
   \begin{align*}
   T_{k+1}
        \leq  & (1-\alpha_k)  T_k +\frac{c_4}{(k+1)^{2p-\tau}}
		+ \left(\frac{1}{\lambda_k}+\alpha^2_k\lambda_k+T_k\right) \frac{c_4}{k+1} \\
        \leq & \left(1-\alpha_k+\frac{c_4}{k+1}\right)  T_k 
        +\frac{c_4+c_5}{(k+1)^{2p-\tau}}
		+\frac{c_5}{(k+1)^{1+\tau}} \\
  \leq & \left(1 - \frac{\alpha_k}{2} \right)  T_k +\frac{c_4+c_5}{(k+1)^{2p-\tau}}
		+\frac{c_5}{(k+1)^{1+\tau}}.
\end{align*}
Since $\min\{ 2p-\tau, 1+\tau\} > p$,
the relation \eqref{prop-rate-T-2} follows from Proposition~\ref{prop_phi1}.

Next, we show \eqref{prop-rate-TG-2}.
We have from \eqref{prop-rate-T-3} that for all $i\ge 0$,
\begin{align*}
&\alpha_i G_{{\bm z}^i,\lambda_i}({\bm x^i}) \leq (\cL_{\lambda_i}({\bm x}^i,{\bm z}^i) -L^*)-(\cL_{\lambda_{i+1}}({\bm x}^{i+1},{\bm z}^{i+1})- L^*) +\beta_i,
   \end{align*}
   where we obtain $\beta_i$ as in \eqref{prop-rate-T-4}.
  Note that we have $|\cL_{\lambda_k}({\bm x}^k,{\bm z}^k)-L^*|\rightarrow 0$ by Theorem~\ref{the-con-0} (this theorem is applicable thanks to $T_k\rightarrow 0$; see \eqref{prop-rate-T-2}).
  Let $k_1\geq k_0$ be a positive integer such that $|\cL_{\lambda_{k}}({\bm x}^{k},{\bm z}^{k})- L^*|\leq 1$ for any $k\geq k_1$.
  Summing both sides of the last display from $i=0$ to some $k$ with $k\geq k_1$, we obtain
\begin{align}\textstyle
\sum_{i=0}^{k} \alpha_i G_{{\bm z}^i, \lambda_i}({\bm x}^{i})\leq (|\cL_{\lambda_{0}}({\bm x}^{0},{\bm z}^{0})-L^*|+1)+ \sum_{i=0}^{k}\beta_i.  \label{prop-rate-T-6}
  \end{align}
Now, from $p>\tau$ (see Assumption~\ref{assu-parameter}(ii)), $\lambda_k=\Theta(k^{\tau})$  (see Assumption~\ref{assu-parameter}(i)), \eqref{prop-rate-T-2} and the second inequality in \eqref{prop-rate-T-5}, we see that the second term in $\beta_k$ is summable. Then, we have
\begin{equation*}
\sum_{i=0}^k\beta_i=\mathcal{O}\bigg(\sum_{i=0}^k\frac{\lambda_i+c_2}{2} c_3 \alpha_i^2\bigg)=
\begin{cases}
\mathcal{O}(1) & \mbox{if $2p-\tau >1$}, \\
\mathcal{O}(\log(k))  & \mbox{if $2p-\tau =1$},  \\
\mathcal{O}(k^{1-(2p-\tau)})  & \mbox{if $2p-\tau <1$}.
\end{cases}
\end{equation*}
In addition, $\Gamma_k=\sum_{i=0}^k \alpha_i=\Theta(k^{1-p})$ with $p\in (0,1)$.
Based on these observations, the relation \eqref{prop-rate-TG-2} now follows immediately upon dividing both sides of \eqref{prop-rate-T-6} by $\Gamma_k$.
\end{proof}

We are now ready to present the following convergence result of Algorithm~\ref{alg} under Assumption~\ref{assu-parameter}.
\begin{theorem}[Global convergence and complexity: open-loop stepsize]\label{the-con}
Suppose that Assumptions~\ref{assu-function} and \ref{assu-parameter} hold.
	Let  $\{({\bm x}^k,{\bm z}^k)\}$ be generated by Algorithm \ref{alg} and $L^*$ be the optimal value of problem \eqref{problem}. 
		Let  $(p,\tau)$ be given in Assumption \ref{assu-parameter}.
	Then
        \begin{equation*}
	\max\{|f({\bm x}^k)+h({\bm x}^k)-L^*|, \|[{\bm g}({\bm x}^k)]_+\|\} = \mathcal{O}(\max\{k^{-(p-\tau)},k^{-\tau} \}).
\end{equation*}
If, in addition, problem \eqref{problem} has a Slater point, then there exists a subsequence of $\{({{\bm x}^k,[\lambda_k {\bm g}({\bm x}^k)+{\bm z}^k]_+})\}$ such that all of its cluster points are KKT points.
\end{theorem}
\begin{proof}
Using \eqref{prop-rate-T-2} and Theorem~\ref{lemma-fea-fun-bound}, we deduce the desired complexity bounds on feasibility violation and primal suboptimality.

Next, since $\sum_{i=0}^k \frac{\alpha_i}{\Gamma_k}G_{{\bm z}^i, \lambda_i}({\bm x}^{i})\rightarrow 0$ (see \eqref{prop-rate-TG-2}),
the second conclusion follows from the last assertion in Theorem~\ref{the-con-0}.
\end{proof}

 \begin{remark}\label{remark-dmin-rate}
By taking $\tau=0.5$, we see from Theorem~\ref{the-con} that the rate of $\max\{|f({\bm x}^k)+h({\bm x}^k)-L^*|, [{\bm g}({\bm x}^k)]_+\|\}\rightarrow 0$ can be made arbitrarily close to $\mathcal{O}(1/\sqrt{k})$ by choosing $p \in (0,1)$ arbitrarily close to $1$. 
 \end{remark}

\subsection{The short stepsize}

In this section, we consider the following assumption for $\{\alpha_k\}$, $\{\lambda_k\}$ and $\{\sigma_k\}$ in Algorithm \ref{alg}.

\begin{assumption}[Assumptions for  parameters in Algorithm \ref{alg}]\label{assu-parameter-ss}
	\begin{enumerate}[{\rm (i)}]
			\item The positive penalty parameter $\{\lambda_k\}$ is increasing, $\lambda_k=\Theta(k^\tau)$ and $\lambda_{k+1} -\lambda_k = \mathcal{O}(k^{-(1-\tau)})$ for some $\tau \in (0,1)$.
			\item  The $\alpha_k$ is the short stepsize in \eqref{short_stepsize_haha} with $({\bm x}^k,{\bm z}^k,{\bm v}^k,\lambda_k)$ in place of $({\bm x},{\bm z},{\bm v}^+,\lambda)$.
            \item The positive sequence  $\{\sigma_k\}$ satisfies  $\sigma_k\leq \lambda_k$ and 
			$\sigma_k=\Theta(k^{-(1+\gamma)})$ for some $\gamma>0$.
	\end{enumerate}
\end{assumption}

It now suffices to study the convergence behaviors of $\{T_k\}$ and $\{G_{{\bm z}^k,\lambda_k}({\bm x}^k)\}$ before we can invoke Theorem~\ref{lemma-fea-fun-bound} and~\ref{the-con-0} to deduce global convergence and complexity. To this end, we show in the next proposition that $\{T_k\}$ satisfies a certain difference inequality and establish the convergence rate of $\{T_k\}$ and a weighted average of $\{(G_{{\bm z}^k,\lambda_k}({\bm x}^k))^2\}$.

\begin{proposition}\label{prop-rate-T-ss}
Suppose that Assumptions~\ref{assu-function} and \ref{assu-parameter-ss} hold.
	Let $\{({\bm x}^k,{\bm z}^k)\}$ be generated by Algorithm \ref{alg} and $L^*$ be the optimal value of problem \eqref{problem}. Let $\cL_\lambda$ be defined in \eqref{lagrangian}, $G_{{\bm z}, \lambda}({\bm x})$ be defined in \eqref{def-G}, and $T_k$ be given in \eqref{def-Tk}.  Then, there exist some positive constants $c_1$, $c_2$, $c_3$ and positive integer $k_0$ such that for all $k\ge k_0$,
		\begin{align}
       & \cL_{\lambda_{k+1}}({\bm x}^{k+1},{\bm z}^{k+1}) -L^*\notag\\
       &\quad\quad\le (\cL_{\lambda_{k}}({\bm x}^{k},{\bm z}^{k}) -L^*) - \frac{1}{4} \alpha_k G_{{\bm z}^k,\lambda_k}({\bm x}^k)+ \frac{c_2}{(k+1)^{1+\tau}} + \frac{c_3}{(k+1)^{2-\tau}}, \label{prop-rate-T-ss-0}\\
&T_{k+1}\leq T_k \max\left\{\frac{3}{4}, 1-\frac{T_k}{4c_1 \lambda_k} \right\} + \frac{c_2}{(k+1)^{1+\tau}}+\frac{c_3}{(k+1)^{2-\tau}}. \label{prop-rate-T-ss-1}
		\end{align}
Moreover, it holds that
	\begin{align}\label{prop-rate-T-ss-2}
    	T_k = \mathcal{O}(\max\{1/\sqrt{k},1/k^{1-\tau}\})
    \,\, \mbox{and}\,\,
   \sum_{i=0}^k\! \frac{\lambda_i^{-1}}{\Lambda_k} (G_{{\bm z}^i, \lambda_i}({\bm x}^{i}))^2 = \mathcal{O}(1/k^{1-\tau}),
	\end{align}
where $\Lambda_k:=\sum_{i=0}^k \lambda_i^{-1}$, and $\tau$ is given in Assumption~\ref{assu-parameter-ss}.
\end{proposition}
\begin{proof}
First, we note that for any $k\ge 0$,
\begin{equation}\label{prop-rate-T-ss-3-0}
(L_\Psi({\bm x}^k,{\bm z}^k)+L_f)\|{\bm x}^k-{\bm v}^k\|^2\leq (D^2 \varsigma + D^2 L_f / \lambda_0) \lambda_k = c_1\lambda_k.
\end{equation}
where the inequality follows from \eqref{Lip-Lpsi}, Assumption~\ref{assu-function}(ii) and the monotonicity of $\{\lambda_k\}$,
and the equality follows upon letting $c_1:= D^2 \varsigma + D^2 L_f / \lambda_0$. 
In addition, using \eqref{lemma-psi-2} in Lemma~\ref{lemma-psi}, the relation $\| [{\bm g}({\bm x}^k)]_+\| = \mathcal{O}(\max\{(T_k/\lambda_k)^\frac12,1/\lambda_k\})$ from Theorem~\ref{lemma-fea-fun-bound}, and the fact that $\sigma_k=\Theta(k^{-1-\gamma})$ with $\gamma>0$, we see that there exists $\hat{c}_2>0$ such that for any $k$
\begin{equation}\label{prop-rate-T-ss-3-1}
\cL_{\lambda_{k+1}}({\bm x}^{k+1},{\bm z}^{k+1})  -\cL_{\lambda_{k}}({\bm x}^{k+1},{\bm z}^{k}) \leq \bigg(\frac{1}{\lambda_k^2}+\alpha_k  G_{{\bm z}^k,\lambda_k}({\bm x}^k)+\frac{T_k}{\lambda_k}\bigg)\frac{\hat{c}_2}{(k+1)^{1-\tau}}.
\end{equation}
From Assumption~\ref{assu-parameter-ss}, we have that $\lambda_k=\Theta(k^{\tau})$ with $\tau\in (0,1)$.
From this, there exists a positive integer $k_0$ such that for any $k\geq k_0$
\begin{equation}\label{prop-rate-T-ss-3}
\begin{split}
\frac{1}{8} \ge \frac{\hat{c}_2}{(k+1)^{1-\tau}}\ \ {\rm and}\ \  0.5  -\frac{\hat{c}_2}{(k+1)^{1-\tau}}-\frac{\hat{c}_2}{\lambda_k (k+1)^{1-\tau}}\ge\frac{1}{4}.
\end{split}
\end{equation}
In addition, fixed some $\epsilon \in (0,1/(8\hat{c}_2c_1)]$, we have that 
\begin{align}\label{prop-rate-T-ss-3-2}
 \frac{1}{8c_1} \ge\hat{c}_2\epsilon.
\end{align}
Now, we claim that it holds for any $k\geq k_0$
\begin{align}
& 0.5 \alpha_k  G_{{\bm z}^k,\lambda_k}({\bm x}^k) - \bigg(\alpha_k  G_{{\bm z}^k,\lambda_k}({\bm x}^k) + \frac{T_k}{\lambda_k}\bigg)\frac{\hat{c}_2}{(k+1)^{1-\tau}}+\frac{(4\epsilon)^{-1}}{\lambda_k}\frac{\hat{c}_2}{(k+1)^{2-2\tau}} \notag \\
\geq & 0.25\alpha_k G_{{\bm z}^k,\lambda_k}({\bm x}^k). \label{prop-rate-T-ss-5}
\end{align}
We will proceed by considering the cases $\alpha_k=0$, $\alpha_k\in (0,1)$ and $\alpha_k =1$, separately.
For the case that $\alpha_k =0$, we have that $G_{{\bm z}^k,\lambda_k}({\bm x}^k)= 0$ and thus $T_k=0$ thanks to Lemma~\ref{lemma-G-T}. It is trivial that \eqref{prop-rate-T-ss-5} holds.
If $\alpha_k\in (0,1)$,  by the definition of the short stepsize and \eqref{prop-rate-T-ss-3-0},  we have
\begin{align}\label{prop-rate-T-ss-4}
\alpha_k = G_{{\bm z}^k,\lambda_k}({\bm x}^k)/((L_\Psi({\bm x}^k,{\bm z}^k)+L_f)\|{\bm x}^k-{\bm v}^k\|^2)\geq G_{{\bm z}^k,\lambda_k}({\bm x}^k)/(c_1 \lambda_k).
\end{align}
Hence, for $\alpha_k\in (0,1)$, it holds that
\begin{align*}
& 0.5 \alpha_k  G_{{\bm z}^k,\lambda_k}({\bm x}^k) - \bigg(\alpha_k  G_{{\bm z}^k,\lambda_k}({\bm x}^k) + \frac{T_k}{\lambda_k}\bigg)\frac{\hat{c}_2}{(k+1)^{1-\tau}}+\frac{(4\epsilon)^{-1}}{\lambda_k}\frac{\hat{c}_2}{(k+1)^{2-2\tau}} \notag \\
=& 0.5 \alpha_k  G_{{\bm z}^k,\lambda_k}({\bm x}^k) - \frac{\hat{c}_2\alpha_k  G_{{\bm z}^k,\lambda_k}({\bm x}^k)}{(k+1)^{1-\tau}} -\frac{T_k}{(k+1)^{1-\tau}} \frac{\hat{c}_2}{\lambda_k}+\frac{1}{4\epsilon \lambda_k} \frac{\hat{c}_2}{(k+1)^{2-2\tau}}\\
\geq& \bigg(0.5 \alpha_k-\frac{\hat{c}_2}{(k+1)^{1-\tau}}\alpha_k- \frac{\hat{c}_2\epsilon G_{{\bm z}^k,\lambda_k}({\bm x}^k)}{\lambda_k}\bigg)  G_{{\bm z}^k,\lambda_k}({\bm x}^k)  \\
\geq& \frac{\alpha_k}{4} G_{{\bm z}^k,\lambda_k}({\bm x}^k) + \bigg(\frac{\alpha_k}{8} + \frac{G_{{\bm z}^k,\lambda_k}({\bm x}^k)}{8c_1 \lambda_k}\!-\!\frac{\hat{c}_2}{(k+1)^{1-\tau}}\alpha_k - \frac{\hat{c}_2\epsilon G_{{\bm z}^k,\lambda_k}({\bm x}^k)}{\lambda_k} \bigg)G_{{\bm z}^k,\lambda_k}({\bm x}^k) \\
\geq & 0.25\alpha_k G_{{\bm z}^k,\lambda_k}({\bm x}^k),
\end{align*}
where the first inequality follows from Lemma~\ref{lemma-G-T} and the fact that 
\begin{align*}
\frac{T_k} {(k+1)^{1-\tau}} = (\sqrt{2\epsilon}T_k) (\sqrt{2\epsilon}^{-1}(k+1)^{\tau-1})\leq \epsilon T_k^2+ \frac{1}{4\epsilon(k+1)^{2-2\tau}},
\end{align*}
  the second inequality follows from \eqref{prop-rate-T-ss-4}, 
  and the last inequality follows from \eqref{prop-rate-T-ss-3} and \eqref{prop-rate-T-ss-3-2}.
If $\alpha_k=1$, then we have
 \begin{align*}
 &  0.5 \alpha_k  G_{{\bm z}^k,\lambda_k}({\bm x}^k) - \bigg(\alpha_k  G_{{\bm z}^k,\lambda_k}({\bm x}^k) + \frac{T_k}{\lambda_k}\bigg)\frac{\hat{c}_2}{(k+1)^{1-\tau}} \\
 \geq & \bigg(0.5  -\frac{\hat{c}_2}{(k+1)^{1-\tau}}-\frac{\hat{c}_2}{\lambda_k (k+1)^{1-\tau}}\bigg) G_{{\bm z}^k,\lambda_k}({\bm x}^k)\geq 0.25\alpha_k G_{{\bm z}^k,\lambda_k}({\bm x}^k),
\end{align*}
where the first inequality follows from Lemma~\ref{lemma-G-T} and the fact that $\alpha_k=1$, and we use \eqref{prop-rate-T-ss-3} and the fact $\alpha_k=1$ for the second inequality. Therefore, we have verified \eqref{prop-rate-T-ss-5}.

  Applying Lemma \ref{lemma-od} with $({\bm x}^{k+1},{\bm x}^k,{\bm z}^k,{\bm v}^k,\alpha_k,\lambda_{k})$ in place of the sextuple $({\bm x}^+,{\bm x},{\bm z},{\bm v}^+,\alpha,\lambda)$ in which $\alpha_k$ is the short stepsize in Assumption \ref{assu-parameter-ss}(ii), we have
   \begin{align}\label{prop-rate-T-ss-6}
\mathcal{L}_{\lambda_k}({\bm x}^{k+1},{\bm z}^k)
\leq & \mathcal{L}_{\lambda_k}({\bm x}^k,{\bm z}^k) - 0.5\alpha_k G_{{\bm z}^k,\lambda_k}({\bm x}^k).
	\end{align}
  For any $k\geq k_0$, it holds that
    \begin{align}
    &(\cL_{\lambda_{k+1}}({\bm x}^{k+1},{\bm z}^{k+1}) -L^*)-(\cL_{\lambda_{k}}({\bm x}^{k},{\bm z}^{k}) -L^* )
   \notag\\
   \overset{\rm (a)}{\leq}   & - 0.5 \alpha_k G_{{\bm z}^k,\lambda_k}({\bm x}^k)+\cL_{\lambda_{k+1}}({\bm x}^{k+1},{\bm z}^{k+1})  -\cL_{\lambda_{k}}({\bm x}^{k+1},{\bm z}^{k}) \notag \\
    \overset{\rm (b)}\leq & - 0.5 \alpha_k G_{{\bm z}^k,\lambda_k}({\bm x}^k) + \bigg(\frac{1}{\lambda_k^2}+\alpha_k  G_{{\bm z}^k,\lambda_k}({\bm x}^k)+\frac{T_k}{\lambda_k}\bigg)\frac{\hat{c}_2}{(k+1)^{1-\tau}}  \notag\\
    = & - 0.5 \alpha_k  G_{{\bm z}^k,\lambda_k}({\bm x}^k) + \bigg(\alpha_k  G_{{\bm z}^k,\lambda_k}({\bm x}^k) + \frac{T_k}{\lambda_k}\bigg)\frac{\hat{c}_2}{(k+1)^{1-\tau}}
    - \frac{(4\epsilon)^{-1}\hat{c}_2}{\lambda_k(k+1)^{2-2\tau}} \notag \\
    &+\frac{(4\epsilon)^{-1}\hat{c}_2}{\lambda_k(k+1)^{2-2\tau}}+\frac{1}{\lambda_k^2}\frac{\hat{c}_2}{(k+1)^{1-\tau}}  \notag  \\ 
     \overset{\rm (c)}\leq & - 0.25 \alpha_k G_{{\bm z}^k,\lambda_k}({\bm x}^k)+ \frac{c_2}{(k+1)^{1+\tau}}+ \frac{c_3}{(k+1)^{2-\tau}}, \notag  
    \end{align}
    where (a) follows from \eqref{prop-rate-T-ss-6},
    (b) follows from \eqref{prop-rate-T-ss-3-1},
    and (c) follows from \eqref{prop-rate-T-ss-5} for some constants $c_2\geq\hat{c}_2$ and $c_3 \geq (4\epsilon)^{-1}\hat{c}_2$ since $\lambda_k=\Theta(k^{\tau})$. This establishes \eqref{prop-rate-T-ss-0}.
    From the last display, when $G_{{\bm z}^k,\lambda_k}({\bm x}^k) > 0$, we have
     \begin{align*}
   &\cL_{\lambda_{k+1}}({\bm x}^{k+1},{\bm z}^{k+1}) -L^*\notag\\
   &\leq   T_k- 0.25 G_{{\bm z}^k,\lambda_k}({\bm x}^k)\min\bigg\{1, \frac{G_{{\bm z}^k,\lambda_k}({\bm x}^k)}{c_1\lambda_k}\bigg\}+ \frac{c_2}{(k+1)^{1+\tau}}+ \frac{c_3}{(k+1)^{2-\tau}}   \notag  \\
     &\leq   T_k \max\left\{\frac{3}{4}, 1-\frac{T_k}{4c_1 \lambda_k} \right\} + \frac{c_2}{(k+1)^{1+\tau}} + \frac{c_3}{(k+1)^{2-\tau}},  
    \end{align*}
    where the first inequality follows from \eqref{prop-rate-T-ss-3-0} and the definition of $T_k$ in \eqref{def-Tk}, and the second inequality follows from Lemma~\ref{lemma-G-T}. The above inequality clearly also holds when $G_{{\bm z}^k,\lambda_k}({\bm x}^k) = 0$ because this means $T_k = 0$ as well thanks to Lemma~\ref{lemma-G-T}.
    The above display together with the definition of $T_k$ proves \eqref{prop-rate-T-ss-1}.

Next, we prove the first relation in \eqref{prop-rate-T-ss-2}. Dividing both sides of \eqref{prop-rate-T-ss-1} by $4c_1\lambda_k$, we have upon using the substitution $\varphi_k := T_k/(4c_1\lambda_k)$ that for all $k\ge k_0$,
\begin{align}
\varphi_{k+1} =  \frac{T_{k+1}}{4c_1\lambda_{k+1}}\le  \frac{T_{k+1}}{4c_1\lambda_{k}} \le \varphi_k \max\bigg\{\frac{3}{4}, 1-\varphi_k\bigg\}  + \frac{c_4}{(k+1)^{1+2\tau}} + \frac{c_5}{(k+1)^{2}},\label{Recur_varphi}
\end{align}
where the first inequality holds because $\{\lambda_k\}$ is nondecreasing, and the last inequality holds for some constants $c_4\geq c_2$ and $c_5\geq c_3$ thanks to $\lambda_k=\Theta(k^{\tau})$.
In view of \eqref{Recur_varphi} and Remark~\ref{comment_condition_gen}, we can now invoke Proposition~\ref{prop_phi2} with $(\eta,\mu) =(3/4, 1)$ and $\gamma_k=c_4/(k+1)^{1+2\tau}+c_5/(k+1)^2$ to deduce that $\varphi_k =\mathcal{O}(\gamma_k^{1/2}) =\mathcal{O}(\max\{1/k^{1/2+\tau}, 1/k\})$. The first estimate in \eqref{prop-rate-T-ss-2} now follows immediately upon recalling $\varphi_k = T_k/(4c_1\lambda_k)$ and $\lambda_k = \Theta(k^\tau)$ (see Assumption~\ref{assu-parameter-ss}(i)).

Finally, we show the second estimate in \eqref{prop-rate-T-ss-2}. From \eqref{prop-rate-T-ss-0},
we have for all $i\ge k_0$,
\begin{align}
&\frac{G_{{\bm z}^i,\lambda_i}({\bm x}^i)}{4}\min\bigg\{1, \frac{G_{{\bm z}^i,\lambda_i}({\bm x}^i)}{c_1 \lambda_i} \bigg\} \notag \\
\leq& (\cL_{\lambda_i}({\bm x}^i,{\bm z}^i) -L^*) - (\cL_{\lambda_{i+1}}({\bm x}^{i+1},{\bm z}^{i+1})- L^*)
     +\frac{c_2}{(i+1)^{1+\tau}} +\frac{c_3}{(i+1)^{2-\tau}}, \label{prop-rate-T-ss-8}
   \end{align}
   where the first inequality follows from \eqref{prop-rate-T-ss-3-0}. 
   Since we have $|\cL_{\lambda_i}({\bm x}^i,{\bm z}^i) -L^*|\rightarrow 0$ thanks to Theorem~\ref{the-con-0} (this theorem is applicable because $T_k\rightarrow 0$; see the first relation in \eqref{prop-rate-T-ss-2}),
   there exists $k_1\geq k_0$ that $|\cL_{\lambda_i}({\bm x}^i,{\bm z}^i) -L^*|\leq  1$ for any  $i\geq k_1$. In addition, since the right-hand side of  \eqref{prop-rate-T-ss-8} goes to zero and $\lambda_i\rightarrow \infty$, we deduce from \eqref{prop-rate-T-ss-8} that there exists $k_2\geq k_1$ such that 
   \[
   G_{{\bm z}^i,\lambda_i}({\bm x}^i)\leq c_1\lambda_i
   \quad \forall\, i\geq k_2.
   \]
From these facts,  summing both sides of \eqref{prop-rate-T-ss-8} from $i=k_2$ to some $k>k_2$, we obtain
   \begin{align}\label{prop-rate-T-ss-9}
   \sum_{i=k_2}^k \frac{(G_{{\bm z}^i,\lambda_i}({\bm x}^i))^2}{4 c_1 \lambda_i}
   \leq 
   (|\cL_{\lambda_0}({\bm x}^0,{\bm z}^0) -L^*|+1)
     +\sum_{i= k_2}^k \left(\frac{c_2}{(i+1)^{1+\tau}}+\frac{c_3}{(i+1)^{2-\tau}}\right).
   \end{align}
   Dividing both sides of \eqref{prop-rate-T-ss-9} by $\Lambda_k=\sum_{i=0}^k \lambda_i^{-1}$ and noting that  $\Lambda_k=\Theta(k^{1-\tau})$ and the right-hand side of \eqref{prop-rate-T-ss-9} is bounded as $k\rightarrow \infty$ thanks to $\tau \in(0,1)$, we deduce that
 \begin{align*}\textstyle
   \Lambda_k^{-1}\sum_{i=k_2}^k \lambda_i^{-1} (G_{{\bm z}^i,\lambda_i}({\bm x}^i))^2
   =\mathcal{O}(1/k^{1-\tau}),
   \end{align*}
   proving the second estimate in \eqref{prop-rate-T-ss-2}.
\end{proof}

We are now ready to present the convergence result for Algorithm~\ref{alg} under Assumption~\ref{assu-parameter-ss}.
\begin{theorem}[Global convergence and complexity: short stepsize]\label{the-con-ss}
Suppose that Assumptions~\ref{assu-function} and \ref{assu-parameter-ss} hold.
	Let $\{({\bm x}^k,{\bm z}^k)\}$ be generated by Algorithm \ref{alg} and $L^*$ be the optimal value of problem \eqref{problem}. 
	Then
     \begin{equation*}
	\max\{|f({\bm x}^k)+h({\bm x}^k)-L^*|, \|[{\bm g}({\bm x}^k)]_+\|\} =\mathcal{O}(\max\{k^{-1/2}, k^{-(1-\tau)}, k^{-\tau}\}). 
\end{equation*}
If, in addition, problem \eqref{problem} has a Slater point, then there exists a subsequence of $\{({{\bm x}^k,[\lambda_k {\bm g}({\bm x}^k)+{\bm z}^k]_+})\}$ such that all of its cluster points are KKT points.
\end{theorem}
\begin{proof}
Using the first estimate in \eqref{prop-rate-T-ss-2} and Theorem~\ref{lemma-fea-fun-bound}, we deduce the desired complexity bounds on feasibility violation and primal suboptimality.

Finally, since $\sum_{i=0}^k \frac{\lambda_i^{-1}}{\Lambda_k}(G_{{\bm z}^i, \lambda_i}({\bm x}^{i}))^2\rightarrow 0$ (see the second relation in \eqref{prop-rate-T-ss-2}), the second conclusion follows from the last assertion in Theorem~\ref{the-con-0}.
\end{proof}

 \begin{remark}\label{remark-short-rate}
Taking $\tau=0.5$, we see from Theorem~\ref{the-con-ss} that Algorithm~\ref{alg} achieves the rate  $\max\{|f({\bm x}^k)+h({\bm x}^k)-L^*|, \|[{\bm g}({\bm x}^k)]_+\|\}=\mathcal{O}(1/\sqrt{k})$.
 \end{remark}

\section{Improved convergence rate under uniform convexity}\label{sec-con-ss}

It is known that the conditional gradient method with either the vanilla stepsize or the short stepsize enjoys a faster convergence rate when the constraint set admits some special structures such as uniform convexity; see, e.g., \cite{Garber-Hazan-2015,Kerdreux-dAspremont-Pokutta-2021, Wirth2025}. Inspired by this, in this section, we analyze Algorithm~\ref{alg} when $h$ is the indicator function of a uniformly convex set and show that one can achieve better convergence rates compared with those in Theorems~\ref{the-con} and \ref{the-con-ss}. To our knowledge, this is the first single-loop CG algorithm that demonstrates acceleration on an intersection of a uniformly convex set with convex inequalities. We first recall the definition of uniformly convex sets.

\begin{definition}[Uniformly convex sets \cite{Garber-Hazan-2015, Kerdreux-dAspremont-Pokutta-2021, Wirth2025,Wirth2023}]\label{def-u-convex}
A convex set $\mathcal{D}\subseteq \R^n$ is  said to be $(\nu,q)$-uniformly convex for some $q\geq 2$ and $\nu>0$
if, for all ${\bm x},{\bm y}\in \mathcal{D}$ and $t\in [0,1]$, it holds that
	\[
	t {\bm x} +  (1-t){\bm y}+  t (1-t)\nu\|{\bm x}-{\bm y}\|^q {\bm z}\in \mathcal{D}
	\]
whenever  ${\bm z}\in \R^n$ with $\|\bm z\|= 1$.
We say that $\mathcal{D}$ is strongly convex when $q=2$. 
\end{definition}


For given $\epsilon>0$, define the set
\begin{equation}\label{def-A}
\mathcal{A}_\epsilon:=\{ ({\bm x},{\bm z})\in {\bm g}^{-1}((-\infty,\epsilon]^m)\times \R^m_+ : \mbox{$\forall i$, if $g_i({\bm x})< -\epsilon$, then $z_i=0$}\}.
\end{equation}

In this section, we consider the following additional assumptions on \eqref{problem}.
\begin{assumption}\label{assu-uc-function}
	The following conditions hold for \eqref{problem}.
\begin{enumerate}[\rm (i)]
	\item Function $h$ is the indicator function of a convex set $\mathcal{C}$, i.e., $h={\rm Ind}_{\mathcal{C}}$, where $\mathcal{C}$ is compact and $(\nu,q)$-uniformly convex for some $q\ge 2$ and $\nu > 0$.
	\item For some $\epsilon >0$, it holds that
	\[\textstyle
	\zeta:=\inf\left\{\left\|\nabla f(\bm x)+ \sum_{i=1}^m  z_i\nabla g_i(\bm x)\right\| : {({\bm x},{\bm z})\in \mathcal{A}_\epsilon}, {\bm x}\in \mathcal{C} \right\}>0.
	\]
\end{enumerate}
\end{assumption}

In the absence of ${\bm g}(\bm x)\leq {\bm 0}$, Assumption \ref{assu-uc-function}(ii) reduces to the standard assumption $\inf_{{\bm x}\in \mathcal{C}}\|\nabla f({\bm x})\|>0$ used in the literature for the study of improved convergence rates of Frank-Wolfe algorithm; see, e.g., \cite{Garber-Hazan-2015, Kerdreux-dAspremont-Pokutta-2021}. The next proposition presents a sufficient condition for Assumption~\ref{assu-uc-function}(ii). Intuitively, condition \eqref{prop-suff-con-1} suggests that the set ${\cal C}$ is not redundant when the right hand sides of the inequality constraints in \eqref{problem} are slightly perturbed.

\begin{proposition}[A sufficient condition for Assumption~\ref{assu-uc-function}(ii)]\label{prop-suff-con}
Let Assumptions~\ref{assu-function} and \ref{assu-uc-function}(i) hold.
Suppose that there exist $\epsilon>0$ and $\tilde{\bm x}$ such that $g_i(\tilde {\bm x})< -\epsilon$ for all $i$, and that for all ${\bm \eta}\in[-\epsilon,\epsilon]^m$, it holds that 
\begin{align}\label{prop-suff-con-1}
\Argmin_{{\bm x}\in \R^n}\{f({\bm x}):\; {\bm g}({\bm x}) \le {\bm \eta}\} \cap {\cal C} = \varnothing. 
\end{align}
Then Assumption~\ref{assu-uc-function}(ii) holds with the above $\epsilon$.
\end{proposition}
\begin{proof}
Suppose to the contrary that the $\zeta$ defined in Assumption~\ref{assu-uc-function}(ii) is zero. Let $\{({\bm x}^k,{\bm z}^k)\}\subset\mathcal{A}_\epsilon$ with $\{{\bm x}^k\}\subset {\cal C}$ be a sequence such that 
\begin{align}\label{prop-suff-con-2}
\textstyle\left\|\nabla f({\bm x}^k)+\sum_{i=1}^m z^k_i \nabla g_i({\bm x}^k)\right\|\rightarrow 0.
\end{align}
We claim that $\{{\bm z}^k\}$ is bounded. To see this, suppose to the contrary that $\{{\bm z}^k\}$ is unbounded. Let $\tilde{\mathcal{K}}$ be an index set such that $\lim_{k\to\infty, k\in {\tilde{\cal K}}}\|{\bm z}^k\| = \infty$ and $\{({\bm x}^k,{\bm z}^k/\|{\bm z}^k\|)\}_{k\in \tilde{\cal K}}$ is a convergent subsequence converging to $(\bar{\bm x},\bar{\bm z})$ for some
$\bar{\bm x}\in \mathcal{C}$ (as $\mathcal{C}$ is compact) and $\bar{\bm z}\geq {\bm 0}$ with $\|\bar{\bm z}\|=1$. Then we deduce from \eqref{prop-suff-con-2} and the definition of ${\cal A}_\epsilon$ that
\[\textstyle
\sum_{i=1}^m \bar z_i \nabla g_i(\bar{\bm x}) = {\bm 0}, \ {\rm and} \ \bar z_j = 0\ {\rm whenever}\ g_j(\bar {\bm x}) < -\epsilon.
\]
The above display together with the convexity of $g_i$ shows that
\begin{align*}
0 &\textstyle= \sum_{i=1}^m \bar z_i \nabla g_i(\bar{\bm x})^T(\tilde{\bm x} - \bar{\bm x})\le \sum_{i=1}^m \bar z_i (g_i(\tilde{\bm x}) - g_i(\bar{\bm x}))\\
&\textstyle<\sum_{i=1}^m \bar z_i (-\epsilon - g_i(\bar{\bm x})) = \sum_{i: g_i(\bar {\bm x}) \ge -\epsilon} \bar z_i (-\epsilon - g_i(\bar{\bm x})) \le 0,
\end{align*}
where the strict inequality holds in view of the fact that $\bar {\bm z}\neq 0$ and the definition of $\tilde{\bm x}$. This contradiction shows that $\{{\bm z}^k\}$ is bounded. 

Let $\{({\bm x}^k,{\bm z}^k)\}_{k\in {\cal K}}$ be a convergent subsequence and let $(\hat{\bm x},\hat{\bm z}):=\lim_{k\rightarrow \infty, k\in \mathcal{K}} ({\bm x}^k,{\bm z}^k)$. Since $\{({\bm x}^k,{\bm z}^k)\}\subset\mathcal{A}_\epsilon$ and $\{{\bm x}^k\}\subset{\cal C}$, from the definition of $\mathcal{A}_\epsilon$ in \eqref{def-A} and the closedness of $\mathcal{C}$, we see that
$\hat{\bm x}\in \mathcal{C}$,
$g_i(\hat{\bm x})\leq \epsilon$ and $\hat{z}_i\geq 0$ for all $i$, and $\hat{z}_j=0$ whenever $g_j(\hat{\bm x})<-\epsilon$.
Set $\eta_i :=  \max\{g_i(\hat{\bm x}),-\epsilon\}$ for each $i$. Then ${\bm \eta}\in [-\epsilon,\epsilon]^m$ and we have 
\begin{equation}\label{display_for_short}
\hat{\bm z}^T ({\bm g}(\hat {\bm x})-{\bm \eta})= 0\ {\rm and}\ {\bm g}(\hat {\bm x})-{\bm \eta}\leq {\bm 0}.
\end{equation}
Moreover, from \eqref{prop-suff-con-2}, we have ${\bm 0}=\nabla f(\hat{\bm x})+\sum_{i=1}^m\hat{z}_i \nabla g_i(\hat{\bm x})$. This together with \eqref{display_for_short} shows that $\hat{\bm x}\in \Argmin_{{\bm x}\in \R^n}\{f({\bm x}):\; {\bm g}({\bm x}) \le {\bm \eta}\}$,
contradicting \eqref{prop-suff-con-1} (as $\hat{\bm x}\in {\cal C}$).
\end{proof}

Let $\hat{\bm v}\in \Argmin_{{\bm v}\in \mathcal{D}} \langle {\bm u}, {\bm v}\rangle$ for some ${\bm u}\in \R^n$.
If $\mathcal{D}$ is $(\nu,q)$-uniformly convex, then it holds that (see \cite[Lemma 2.1]{Kerdreux-dAspremont-Pokutta-2021})
\begin{align}\label{lb}
	\langle - {\bm u}, \hat{\bm v}-{\bm v}  \rangle \geq (\nu/2)  \|\hat{\bm v}- {\bm v}\|^q \|{\bm u}\|\quad \forall\, {\bm v}\in \mathcal{D}.
\end{align}
Then, we deduce from \eqref{lb} and \eqref{def-G} (with $h={\rm Ind}_{\cal D}$) that for all $({\bm x},{\bm z})\in {\cal D}\times \R^m_+$ and $\lambda > 0$,
\begin{align}\label{lb-1}
G_{{\bm z}, \lambda}(\bm x) \geq  (\nu/2) \|{\bm v}^+- {\bm x}\|^q\|\nabla_{\bm x} F_\lambda({\bm x},{\bm z})\|
\end{align}
whenever ${\bm v}^+\in \Argmin_{{\bm v}\in {\cal D}}   \langle \nabla_{\bm x} F_{\lambda}({\bm x}, {\bm z}),{\bm v}\rangle$.

\begin{lemma}[Lower bound on $\|\nabla_{\bm x}F_{\lambda}\|$]\label{lemma-F-bound}
Suppose that Assumptions~\ref{assu-function} and \ref{assu-uc-function} hold, and either Assumption~\ref{assu-parameter} or Assumption~\ref{assu-parameter-ss} holds.
	Let $\{({\bm x}^k,{\bm z}^k)\}$ be the sequence generated by Algorithm \ref{alg}.
	Then, there exists a positive $k_0$ such that
    $
    \inf_{k\geq k_0}\|\nabla_{\bm x} F_{\lambda_k}({\bm x}^k,{\bm z}^k)\|\geq \zeta
    $,
    where $\zeta>0$ is given in Assumption \ref{assu-uc-function}(ii).
	\end{lemma}
	\begin{proof}
    From \eqref{lagrangian-F}, we see that
    		\[\textstyle
		\nabla_{\bm x} F_\lambda({\bm x},{\bm z})= \nabla f(\bm x) +\sum_{i=1}^m\lambda[ g_i(\bm x)+{\bm z}/\lambda]_+\nabla g_i(\bm x). 
		\]
          From Theorem \ref{the-con} (if Assumption~\ref{assu-parameter} holds) and Theorem \ref{the-con-ss} (if Assumption~\ref{assu-parameter-ss} holds), the boundedness of $\{{\bm z}^k\}$ (see \eqref{remark-bound-zk}) and $\lambda_k\rightarrow \infty$, we see that 
        \begin{align*}
         [g_i({\bm x}^k)]_+ \rightarrow 0, \quad z^k_i/\lambda_k\rightarrow 0  \quad \mbox{for each $i\in [m]$}.
        \end{align*}
        Then, for the $\epsilon$ given in Assumption~\ref{assu-uc-function}(ii), there exists a positive $k_0$ such that for all $k\geq k_0$ and for each $i\in[m]$, we have $|z_i^k/\lambda_k|\le 0.5\epsilon$ and $[g_i({\bm x}^k)]_+ \le \epsilon$. Consequently, all $k\geq k_0$ and for each $i\in[m]$, it holds that
        \begin{itemize}
      \item if $g_i({\bm x}^k)\geq -\epsilon$, then it holds that $g_i({\bm x}^k)\leq \epsilon$ and $[g_i({\bm x}^k)+z^k_i/\lambda_k]_+\geq 0$;
    \item if $g_i({\bm x}^k)< -\epsilon$, then $g_i({\bm x}^k)+z^k_i/\lambda_k\leq -0.5 \epsilon$ and hence $[g_i({\bm x}^k)+z^k_i/\lambda_k]_+= 0$.
        \end{itemize}
         This shows that $({\bm x}^k,\lambda_k [{\bm g}({\bm x}^k)+{\bm z}^k/\lambda_k]_+)\in \mathcal{A}_\epsilon$ for all $k\geq k_0$, which implies the desired assertion in view of Assumption \ref{assu-uc-function}(ii) and the fact that ${\bm x}^k\in \mathcal{C}$ (see \eqref{alg-1}).
\end{proof}

Recall that the convergence rate in Theorem~\ref{lemma-fea-fun-bound} is largely governed by the convergence rates of $\{T_k\}$ and $\{1/\lambda_k\}$. We will establish improved convergence rates for $\{T_k\}$, under the two different stepsize rules studied in Section~\ref{sec4}.

\begin{proposition}[Improved rates on uniformly convex sets for $\{T_k\}$ using open-loop stepsize]\label{prop-uc-dmin}
Suppose that Assumptions~\ref{assu-function}, \ref{assu-parameter} and \ref{assu-uc-function} hold. 
Let $\{({\bm x}^k,{\bm z}^k)\}$ be generated by Algorithm \ref{alg} and $T_k$ be given in \eqref{def-Tk}.
Then it holds that
\begin{equation}\label{prop-uc-dmin-1}
T_k=
\begin{cases}
\mathcal{O}\big(\frac{1}{k^{1 -p + \tau}}\big) & \mbox{if $\mu = 0$},\\
\mathcal{O}\big(\max\big\{\frac{1}{k^{(1-\mu)\omega+p-\tau}}, \frac{1}{k^{1-p+\tau}}\big\}\big)  & \mbox{if $\mu \in (0,1)$},
\end{cases}
\end{equation}
where $\omega:=\min\{p-\tau,1-p+\tau\}$ with $0<\tau<p<1$ and $\mu:=1-2/q$.
	\end{proposition}
\begin{proof}
From Assumption \ref{assu-parameter} and Lemma \ref{lemma-F-bound}, there exist $\hat c_1>0$ and positive integer $k_0$ such that 
\begin{align}\label{prop-uc-dmin-2}
\begin{split}
\|\nabla_{\bm x} F_{\lambda_k}({\bm x}^k,{\bm z}^k)\|\geq \zeta\ \ {\rm and}\ \
(\sigma_k +\lambda_{k+1}-\lambda_k)/\lambda_k\leq \hat c_1/(k+1),
\end{split}
\end{align}
for all $k\geq k_0$.
It follows that for all $k\geq k_0$
\begin{align}
		&\cL_{\lambda_{k+1}}({\bm x}^{k+1},{\bm z}^{k+1})-L^* -\bigg(\frac{1}{\lambda_k}+\alpha^2_k\lambda_k+T_k\bigg)\frac{c_2}{k+1} \notag \\
        \overset{\rm (a)}{\leq}& \cL_{\lambda_k}({\bm x}^k,{\bm z}^k)-L^* - \alpha_k G_{{\bm z}^k,\lambda_k}({\bm x}^k) + \frac{L_{\Psi}({\bm x}^k,{\bm z}^k)+L_f}{2}\alpha_k^2 \|{\bm v}^k-{\bm x}^k\|^2 \notag\\
      \overset{\rm (b)}{\leq} & T_k - \alpha_k  G_{{\bm z}^k,\lambda_k}({\bm x}^k) + \frac{\ell \lambda_k}{2}\alpha_k^2 (\|{\bm v}^k-{\bm x}^k\|^{q})^{\frac{2}{q}}\notag \\
      \overset{\rm (c)}{\leq} & T_k - \alpha_k G_{{\bm z}^k,\lambda_k}({\bm x}^k) + \frac{\ell \lambda_k}{2}\alpha_k^2 \bigg(\frac{2}{\nu}\|\nabla_{\bm x}F_{\lambda_k}({\bm x}^k,{\bm z}^k)\|^{-1}G_{{\bm z}^k,\lambda_k}({\bm x}^k)\bigg)^{\frac{2}{q}}\notag \\
     \overset{\rm (d)}{\leq} & T_k - \alpha_k\big( G_{{\bm z}^k,\lambda_k}({\bm x}^k)- (\varrho/2) \lambda_k\alpha_k  G_{{\bm z}^k,\lambda_k}({\bm x}^k)^{\frac{2}{q}}\big), \label{prop-uc-dmin-3-0}
	\end{align}
    where (a) holds for some $c_2\geq \hat c_1$ by \eqref{prop-rate-T-1} and the second inequality in \eqref{prop-uc-dmin-2},
    (b) holds by invoking the definition of $T_k$ in \eqref{def-Tk} and the fact that $L_{\Psi}({\bm x}^k,{\bm z}^k)+L_f\leq \ell \lambda_k$ for some $\ell>0$ (see \eqref{Lip-Lpsi}),
    (c) follows from \eqref{lb-1},
    (d) follows from the first inequality in \eqref{prop-uc-dmin-2} upon defining $\varrho:=\ell(2/(\nu\zeta))^{\frac{2}{q}}$.
    Since $\alpha^k=\Theta(k^{-p})$ with $p\in (0,1)$, there exists $k_1\geq k_0$ such that $1/2-c_2/[\alpha_k (k+1)]\geq 0$ for any $k\geq k_1$.
 It follows that for any $k\geq k_1$
   \begin{align}
		&\cL_{\lambda_{k+1}}({\bm x}^{k+1},{\bm z}^{k+1})-L^*  \notag \\
    \overset{\rm (a)}{\leq} & T_k - \alpha_k\big( G_{{\bm z}^k,\lambda_k}({\bm x}^k)- (\varrho/2) \lambda_k\alpha_k  G_{{\bm z}^k,\lambda_k}({\bm x}^k)^{\frac{2}{q}}\big)
      +  \bigg(\frac{1}{\lambda_k}+ \alpha^2_k\lambda_k+T_k\bigg) \frac{c_2}{k+1} \notag \\
         \overset{\rm (b)}{\leq} & T_k - \alpha_k\bigg( G_{{\bm z}^k,\lambda_k}({\bm x}^k)- \frac{\varrho \lambda_k\alpha_k }{2}  G_{{\bm z}^k,\lambda_k}({\bm x}^k)^{\frac{2}{q}} - \frac{c_2 G_{{\bm z}^k,\lambda_k}({\bm x}^k)}{\alpha_k (k+1) }\bigg)
      + \frac{c_3}{(k+1)^{1+\tau}} \notag \\
       \overset{\rm (c)}{\leq} & T_k - \alpha_k\bigg( \frac{1}{2}G_{{\bm z}^k,\lambda_k}({\bm x}^k)- \frac{\varrho \lambda_k\alpha_k }{2}  G_{{\bm z}^k,\lambda_k}({\bm x}^k)^{\frac{2}{q}}  \bigg)
      + \frac{c_3}{(k+1)^{1+\tau}},  \label{prop-uc-dmin-3}
	\end{align}
    where (a) follows from \eqref{prop-uc-dmin-3-0}, (b) holds for some $c_3 \geq c_2$ by Assumption~\ref{assu-parameter} (note $1+2p-\tau>1+p>1+\tau$ so $(\lambda_k^{-1} + \alpha_k^2\lambda_k)/(k+1)=\mathcal{O}(1/k^{1+\tau})$ ) and Lemma~\ref{lemma-G-T}, and (c) follows from the facts that $1/2-c_2/[\alpha_k (k+1)]\geq 0$ for any $k\geq k_1$.
    
   We consider the following two cases for $q$.
   \begin{enumerate}[\rm (i)]
   \item $q=2$. In this case,  $\mu =0$. By choosing a larger $k_1$ if necessary, we assume also that $(\varrho/2)\lambda_k \alpha_k\!\le\! 0.25$ for all $k\ge k_1$ (thanks to $p>\tau$). Using this, \eqref{prop-uc-dmin-3} and Lemma~\ref{lemma-G-T}, it holds that
      \begin{align*}\textstyle
     T_{k+1}\leq T_k \left(1-0.25\alpha_k\right)+ {c_3}/{(k+1)^{1+\tau}}\quad  \forall\, k\geq k_1.
     \end{align*}
     Since $1+\tau>p$, invoking Proposition \ref{prop_phi1} and Assumption~\ref{assu-parameter}(ii), we deduce \eqref{prop-uc-dmin-1} for the case $\mu=0$.
  
   \item $q>2$. In this case, we have $\mu =1 - 2/q\in (0,1)$. For any $k\ge k_1$, we proceed by considering two cases. 
   \begin{itemize}
   \item If $G_{{\bm z}^k,\lambda_k}({\bm x}^k) > \varrho^{1/\mu}(\lambda_k \alpha_k)^{1/\mu}$, then in particular $G_{{\bm z}^k,\lambda_k}({\bm x}^k) > 0$ and we have
    $1- \varrho \lambda_k\alpha_k  G_{{\bm z}^k,\lambda_k}({\bm x}^k)^{-\mu}\in (0,1)$. Hence,
    \begin{align}
    &\cL_{\lambda_{k+1}}({\bm x}^{k+1},{\bm z}^{k+1})-L^*\notag\\
    \overset{\rm (a)}\le& T_k -0.5 \alpha_k\big( 1- \varrho \lambda_k\alpha_k  G_{{\bm z}^k,\lambda_k}({\bm x}^k)^{\frac{2}{q}-1}\big)G_{{\bm z}^k,\lambda_k}({\bm x}^k)+ \frac{c_3}{(k+1)^{1+\tau}} \notag \\
    \overset{\rm (b)}{\leq}& T_k \bigg(1-0.5\alpha_k\bigg( 1- \frac{\varrho\lambda_k\alpha_k}{(G_{{\bm z}^k,\lambda_k}({\bm x}^k))^{\mu}}\bigg)\bigg)+ \frac{c_3}{(k+1)^{1+\tau}} \notag \\
    \overset{\rm (c)}{\leq}& T_k(1-0.5\alpha_k) +(\varrho/2)\lambda_k \alpha_k^2 T_k^{1-\mu}+\frac{c_3}{(k+1)^{1+\tau}}
    \notag \\
    \overset{\rm (d)}{\leq}& T_k(1-0.5 \alpha_k) \!+\!\frac{\varrho}{2} \frac{c_4}{(k+1)^{(1-\mu)\omega+(2p-\tau)}} \!+\! \frac{c_3}{(k+1)^{1+\tau}},
    \label{prop-uc-dmin-4} 
    \end{align}
where (a) follows from \eqref{prop-uc-dmin-3}, (b) follows from Lemma~\ref{lemma-G-T} and the fact that $\alpha_k(1- \varrho \lambda_k\alpha_k  G_{{\bm z}^k,\lambda_k}({\bm x}^k)^{-\mu})\in (0,1)$,
(c) follows from Lemma \ref{lemma-G-T},
and (d)  holds for some positive constant $c_4$ in view of \eqref{prop-rate-T-2} (note that $\omega=\min\{p-\tau,1-p+\tau\}$), the facts that $\lambda_k=\Theta(k^\tau)$ and $\alpha_k=\Theta(k^{-p})$.

\item If  $G_{{\bm z}^k,\lambda_k}({\bm x}^k)\leq \varrho^{1/\mu}(\lambda_k \alpha_k)^{1/\mu}$, then we have from Lemma \ref{lemma-G-T} that
\begin{align}\label{prop-uc-dmin-5}
T_k\leq G_{{\bm z}^k,\lambda_k}({\bm x}^k)\leq \varrho^{1/\mu}(\lambda_k \alpha_k)^{1/\mu}=\mathcal{O}(1/k^{(p-\tau)/\mu}),
\end{align}
where the last relation follows from Assumption~\ref{assu-parameter}(i) and (ii). Hence,
\begin{align}
&\cL_{\lambda_{k+1}}({\bm x}^{k+1},{\bm z}^{k+1})-L^*\notag\\
\overset{\rm (a)}\le& T_k \!-\! 0.5 \alpha_kG_{{\bm z}^k,\lambda_k}({\bm x}^k) + (\varrho/2) \lambda_k\alpha_k^2  G_{{\bm z}^k,\lambda_k}({\bm x}^k)^{\frac{2}{q}}+ \frac{c_3}{(k+1)^{1+\tau}}  \notag\\
\overset{\rm (b)}{\leq}& T_k(1-0.5\alpha_k)+ (\varrho/2) \lambda_k\alpha_k^2  G_{{\bm z}^k,\lambda_k}({\bm x}^k)^{1-\mu} +\frac{c_3}{(k+1)^{1+\tau}}\notag \\
\overset{\rm (c)}{\leq}& T_k(1-0.5\alpha_k)+ \frac{c_5}{(k+1)^{\frac{p-\tau}{\mu}+p}}+\frac{c_3}{(k+1)^{1+\tau}}, \label{prop-uc-dmin-6}
\end{align}
where (a) follows from \eqref{prop-uc-dmin-3}, 
(b) follows from Lemma \ref{lemma-G-T},
and (c) holds for some $c_5>0$ by  \eqref{prop-uc-dmin-5} and the facts that $\lambda_k=\Theta(k^\tau)$ and $\alpha_k=\Theta(k^{-p})$.
\end{itemize}

Noting that $1/\mu > 2-\mu$ for $\mu\in(0,1)$ and $\omega=\min\{p-\tau,1-p+\tau\}$, we have 
\[
(1-\mu)\omega+(2p-\tau) \leq (2-\mu)(p-\tau)+p < (p-\tau)/\mu +p.
\]
This fact together with \eqref{prop-uc-dmin-4} and \eqref{prop-uc-dmin-6} implies that there exists $c_6\!\geq\! \max\{c_4,c_5\}$  and $k_2\geq k_1$ such that 
\begin{align*}
T_{k+1}\leq T_k(1-0.5\alpha_k)+ \frac{c_6}{(k+1)^{(1-\mu)\omega+2p-\tau}}  + \frac{c_4}{(k+1)^{1+\tau}} \quad \forall\, k\geq k_2.
\end{align*}
The second bound in \eqref{prop-uc-dmin-1} now follows from Proposition \ref{prop_phi1}.
\end{enumerate}
\end{proof}

\begin{remark}[Improved rate to arbitrarily close to $\mathcal{O}(1/k)$ for open-loop stepsize] \label{remark-uc-dmin-rate}
Clearly, for any $\mu\in[0,1)$, the rates for $T_k$ in \eqref{prop-uc-dmin-1} are never worse than those in Proposition~\ref{prop-rate-T}, and as $\mu$ approaches $1$ (i.e., the set becoming less uniformly convex), we see that $\max\{\frac{1}{k^{(1-\mu)\omega+(p-\tau)}},\frac{1}{k^{\tau}}\}$ becomes $\max\{\frac{1}{k^{p-\tau}},\frac{1}{k^{\tau}}\}$, matching the rate in Proposition~\ref{prop-rate-T}.

Moreover, using Theorem~\ref{lemma-fea-fun-bound}, \eqref{prop-uc-dmin-1} and the fact that $1-p+\tau>\tau$, we deduce that
\begin{align*}
\max\{|f({\bm x}^k)  +  h({\bm x}^k) - L^*|, \|[{\bm g}({\bm x}^k)]_+\|\} 
=
\begin{cases}
 \mathcal{O}(k^{-\tau}) & \mbox{if $\mu=0$}, \\
 \mathcal{O}(\max\{\frac{1}{k^{(1-\mu)\omega+(p-\tau)}},\frac{1}{k^{\tau}}\}) & \mbox{if $\mu\in(0,1)$}. 
 \end{cases}
\end{align*}
Hence, we obtain better rates compared with the rates in Theorem~\ref{the-con}.
 In particular, for $\mu=0$ (i.e. $h={\rm Ind}_\mathcal{C}$ with strongly convex $\mathcal{C}$), the above rate can be made arbitrarily close to $\mathcal{O}(1/k)$ by choosing $p>\tau$ with $\tau \in (0,1)$ being arbitrarily close to $1$.
\end{remark}


\begin{proposition}[Improved rates on uniformly convex sets for $\{T_k\}$ using short stepsize]\label{prop-uc-bound}
Suppose that Assumptions~\ref{assu-function}, \ref{assu-parameter-ss} and \ref{assu-uc-function} hold.
	Let $\{({\bm x}^k,{\bm z}^k)\}$ be generated by Algorithm \ref{alg} and $T_k$ be given in \eqref{def-Tk}.
	Let $\mu:=1-2/q\in[0,1)$. Then it holds 
\begin{equation}\label{prop-uc-bound-1}
T_k =\mathcal{O}(\max\{1/{k^\frac{1}{1+\mu}},1/k^{\frac{2-2\tau}{1+\mu}}\}).
\end{equation}
	\end{proposition}
\begin{proof}
In view of Lemma \ref{lemma-F-bound}, \eqref{Lip-Lpsi} and the monotonicity of $\{\lambda_k\}$, there exists a positive integer $k_0$ such that for any $k\geq k_0$
\begin{align}\label{prop-uc-bound-2}
\frac{\left(\frac{\nu}{2} \|\nabla_{\bm x} F_{\lambda_k}({\bm x}^k,{\bm z}^k)\|\right)^{\frac{2}{q}}}{L_{\Psi}({\bm x}^k,{\bm z}^k)+L_f} 
\geq \frac{(\nu/2)^{2/q}\zeta^{2/q}}{\varsigma \lambda_k+L_f}
\geq \frac{(\nu/2)^{2/q}\zeta^{2/q} }{\varsigma +L_f/\lambda_0}\lambda_k^{-1} =\varrho \lambda_k^{-1},
\end{align}
where the equality holds upon letting $\varrho := \frac{(\nu/2)^{2/q}\zeta^{2/q} }{\varsigma +L_f/\lambda_0}$.
Also, it holds that for all $k\geq k_0$, when $G_{{\bm z}^k,\lambda_k}({\bm x}^k)>0$, we have
    \begin{align}
    & (\cL_{\lambda_{k+1}}({\bm x}^{k+1},{\bm z}^{k+1}) -L^*)-(\cL_{\lambda_{k}}({\bm x}^{k},{\bm z}^{k}) -L^*)\notag \\
     \overset{\rm (a)}{\leq} &  - \frac{1}{4}\alpha_k G_{{\bm z}^k,\lambda_k}({\bm x}^k) + \frac{c_2}{(k+1)^{1+\tau}} +  \frac{c_3}{(k+1)^{2-\tau}}\notag\\
     \overset{\rm (b)}{\leq} & - \frac{T_k}{4}\min\bigg\{1, \frac{(G_{{\bm z}^k,\lambda_k}({\bm x}^k))^{\frac{2}{q}}(G_{{\bm z}^k,\lambda_k}({\bm x}^k))^{1-\frac{2}{q}}}{(L_{\Psi}({\bm x}^k,{\bm z}^k)+L_f)\| {\bm x}^k - {\bm v}^{k}\|^2} \bigg\} +\frac{c_2}{(k+1)^{1+\tau}} +  \frac{c_3}{(k+1)^{2-\tau}}\notag\\
    \overset{\rm (c)}{\leq} &  - \frac{T_k}{4}\min\bigg\{1, \frac{(\frac{\nu}{2} \|\nabla_{\bm x} F_{\lambda_k}({\bm x}^k,{\bm z}^k)\|)^{\frac{2}{q}}(G_{{\bm z}^k,\lambda_k}({\bm x}^k))^{1-\frac{2}{q}}}{L_{\Psi}({\bm x}^k,{\bm z}^k)+L_f} \bigg\} +\frac{c_2}{(k+1)^{1+\tau}} +  \frac{c_3}{(k+1)^{2-\tau}}\notag\\
    \overset{\rm (d)}{\leq} &  - T_k\min\{1/4, (\varrho/4){T_k}^{1-\frac{2}{q}}\lambda_k^{-1}\}+ \frac{c_2}{(k+1)^{1+\tau}}+  \frac{c_3}{(k+1)^{2-\tau}}, \notag
   \end{align}
   where (a) holds for some positive $c_2$ and $c_3$ by \eqref{prop-rate-T-ss-0} in Proposition~\ref{prop-rate-T-ss},
   (b) follows from Lemma~\ref{lemma-G-T},
   (c) follows from \eqref{lb-1} with $({\bm x}^k,{\bm z}^k,{\bm v}^k,\lambda_k)$ in place of $({\bm x},{\bm z},{\bm v}^+,\lambda)$,
  and (d) follows from \eqref{prop-uc-bound-2} and Lemma~\ref{lemma-G-T}. Notice that the above display also holds when $G_{{\bm z}^k,\lambda_k}({\bm x}^k)=0$ (with the convention $0^0 = 1$ when $q = 2$) because we have $T_k = 0$ in view of Lemma~\ref{lemma-G-T}, which implies that the above display (i.e., (a) through (d)) holds because (a) holds regardless of whether $G_{{\bm z}^k,\lambda_k}({\bm x}^k)=0$.

 From the definition of $T_k$ in \eqref{def-Tk}, we obtain from the last display that
\begin{align}
T_{k+1} \leq T_k\max\{3/4, 1- (\varrho/4) T_k^{1-\frac{2}{q}}\lambda_k^{-1}\}+c_2/(k+1)^{1+\tau}\!\!+\!c_3/(k+1)^{2-\tau}\ \ \, \forall k\ge k_0.\label{prop-uc-bound-3}
\end{align}

\begin{enumerate}[\rm (i)]
\item First, we consider $q=2$ so that $\mu := 1-2/q = 0$.  In this case, since $\lambda_k\rightarrow \infty$, by increasing $k_0$ if necessary, we assume that $1- (\varrho/4)\lambda_k^{-1} \geq 3/4$ for all $k\geq k_0$.
Then, the inequality in \eqref{prop-uc-bound-3} reduces to
\begin{align*}
	T_{k+1} \le \left(1- (\varrho/4)\lambda_k^{-1}\right) T_k
	+  c_2/(k+1)^{1+\tau}+c_3/(k+1)^{2-\tau}\ \ \ \forall k\ge k_0.
	\end{align*}
Using Proposition~\ref{prop_phi1}, we can show \eqref{prop-uc-bound-1} for the case $\mu=0$.

	\item  Now, we consider $q>2$.
    Then $\mu := 1-2/q \in (0,1)$. Using the substitution $\phi_k := (\varrho/4)^{1/\mu}\cdot T_k/\lambda_k^{1/\mu}$, we obtain from \eqref{prop-uc-bound-3} that for all $k\ge k_0$,
    \begin{align*}\label{prop-uc-bound-711}
		&\phi_{k+1} = (\varrho/4)^{1/\mu} T_{k+1}/\lambda_{k+1}^{1/\mu}\le (\varrho/4)^{1/\mu} T_{k+1}/\lambda_{k}^{1/\mu}\\
        &
        \leq (\varrho/4)^{1/\mu} T_{k}\lambda_{k}^{-1/\mu}\!\max\{3/4, 1\!-\!(\varrho/4) T_k^\mu\lambda_k^{-1}\} \\
        &+(\varrho/4)^{1/\mu} c_2 (k+1)^{-1-\tau}\lambda_k^{-1/\mu}+(\varrho/4)^{1/\mu} c_3 (k+1)^{-2+\tau}\lambda_k^{-1/\mu}\\
        & \le \phi_k \max\{3/4, 1-\phi_k^\mu\} + (\varrho/4)^{1/\mu} c_4(k+1)^{-1-\tau-\tau/\mu}+ (\varrho/4)^{1/\mu} c_5(k+1)^{-2+\tau-\tau/\mu},
	\end{align*}
    where the first inequality follows from the monotonicity of $\{\lambda_k\}$, and the last inequality holds for some $c_4\geq c_2$ and $c_5\geq c_3$ thanks to $\lambda_k=\Theta(k^{\tau})$. Also, we see that for any $\mu\in(0,1)$,
    \begin{equation*}
    \begin{cases}
    1+\tau < 2-\tau, \,\, 1+\tau+\tau/\mu \leq  1+0.5(\mu+1)/\mu \leq 1+1/\mu & \mbox{if $\tau\in (0,0.5)$}, \\
    1+\tau \geq 2-\tau, \,\, 2-\tau +\tau/\mu =  2+(1/\mu-1)\tau < 1+1/\mu & \mbox{if $\tau\in [0.5,1)$}. \\
    \end{cases}
    \end{equation*}
    Thus, Proposition~\ref{prop_phi2} is applicable in view of Remark~\ref{comment_condition_gen}. Then we have
\begin{align*}
T_k =& \mathcal{O}(\phi_k \lambda_k^{1/\mu}) = \mathcal{O}(\max\{1/k^{(1+\tau+\tau/\mu)/(1+\mu)-\tau/\mu},1/k^{(2-\tau+\tau/\mu)/(1+\mu)-\tau/\mu}\}) \\
=&\mathcal{O}(\max\{1/k^{1/(1+\mu)},1/k^{(2-2\tau)/(1+\mu)}\}).
\end{align*}  	
\end{enumerate}
This completes the proof.
\end{proof}

\begin{remark}[Improved rate to $\mathcal{O}(1/k^{2/3})$ for short stepsize]\label{remark-uc-ss-rate}
We can see that, for any $\mu\in[0,1)$, the rates for $T_k$ in \eqref{prop-uc-bound-1} are never worse than those in Proposition~\ref{prop-rate-T-ss}. In addition,  as $\mu$ approaches $1$ (i.e., the set becomes less uniformly convex), we have that $\max\{1/{k^\frac{1}{1+\mu}},1/k^{\frac{2-2\tau}{1+\mu}}\}$ becomes $\max\{\frac{1}{\sqrt{k}},\frac{1}{k^{1-\tau}}\}$, matching the rate in Proposition~\ref{prop-rate-T-ss}.

Moreover, from Theorem~\ref{lemma-fea-fun-bound}, \eqref{prop-uc-bound-1}, we deduce that
\begin{align*}
\max\{|f({\bm x}^k)  +  h({\bm x}^k) - L^*|, \|[{\bm g}({\bm x}^k)]_+\|\} 
=\mathcal{O}(\max\{1/{k^\frac{1}{1+\mu}},1/k^{\frac{2-2\tau}{1+\mu}},1/k^\tau\}).
\end{align*}
Hence, we obtain better rates compared with the rates in Theorem~\ref{the-con-ss}.
 In particular, for $\mu=0$ (i.e. $h={\rm Ind}_\mathcal{C}$ with strongly convex $\mathcal{C}$), the best rate is $\mathcal{O}(1/k^{2/3})$
by choosing $\tau = 2/3$.
\end{remark}

\section{Preliminary numerical results}\label{sec-numerics}

We test quadratically constrained quadratic program with the Birkhoff polytope constraint as follows:
\begin{equation}
\label{e:num}
\begin{aligned}
&\min_{{\bm x}\in \R^{n\times n }}\, f(\bm x)=\langle {\bm x}- {\bm b}, {\bm A}({\bm x}-{\bm b})\rangle_F + {\rm Ind}_{\mathcal{C}}({\bm x}) \\
&  \quad \mbox{s.t. $g_i({\bm x})= \langle {\bm x},{\bm Q}_i{\bm x}\rangle_F+\langle {\bm x},{\bm r}_i\rangle_F + d_i \leq 0$},  \quad i \in [m],
\end{aligned}
\end{equation}
 where $\langle\cdot,\cdot\rangle_F$ is the trace (Frobenius) inner product,
$\mathcal{C}$ is the set of doubly-stochastic $n\times n$ matrices, also known as the Birkhoff polytope defined as $\mathcal{C}:=\{ {\bm x}\in \R^{n\times n} :  {\bm x} {\bm 1} = {\bm 1}, \,{\bm x}^T {\bm 1} = {\bm 1}, \,{\bm x}\geq {\bm 0} \}$ with ${\bm 1}$ being the vector in $\R^n$ with all elements being $1$, ${\bm A}\in\R^{n\times n}$ is positive definite with eigenvalues selected uniformly at random in the range $\{1,\ldots,10\}$, $m\le n!$ and for every $i\in\{1,\ldots,m\}$, ${\bm Q_i}\in\R^{n\times n}$ is positive definite with eigenvalues being selected uniformly at random in the range $\{0.1,0.2,\ldots,0.9,1\}$, and the entries of ${\bm r}_i\in \R^{n\times n}$ are generated uniformly in $[0,1]$, and the scalar $d_i$ is selected to guarantee that:
\begin{enumerate}[\rm (i)]
\item The barycenter of $\mathcal{C}$, denoted by ${\bm c}$, resides in the relative interior of the constraint, i.e., $g_i(\bm c)< 0$ for each $i$; and
\item At least one distinct vertex of the Birkhoff polytope ${\bm p}_i\in \mathcal{C}$ is excluded from the constraint, i.e., $g_i({\bm p}_i)>0$ for each $i$. 
\end{enumerate}
Together, items (i) and (ii) above guarantee that this instance of \eqref{problem} is feasible and that each constraint is not redundant. Finally, to explore behavior where constraints may interact, the vertex of the quadratic ${\bm b} := {\bm c}+10({\bm p}_1-{\bm c})$ is outside of $\mathcal{C}$ in the direction of an excluded vertex.


We compare the performance of our Algorithm~\ref{alg} with several state-of-the-art  solvers, including the CoexDurCG in \cite{LanR21} and the powerful commercial solver Gurobi \cite{Gurobi} (with an academic license). All algorithms are implemented in Julia (version 1.12.5) on a laptop with Apple M4 processors and 16GB memory.

The parameters for CoexDurCG were selected according to \cite[Corollary~3.4]{LanR21}.
For our Algorithm~\ref{alg}, we set  $\lambda_k=(k+1)^{0.4}$, $\sigma_k = 1/(k+2)^{1+0.01}$, while  the setting for stepsize $\alpha_k$ is as follows:
\begin{itemize}
  \item  open-loop (OL):  $\alpha_k=1/(k+1)^{0.95}$;
  \item short stepsize (SS): a scaled variant of the short stepsize \eqref{short_stepsize_haha} with prefactor $\iota(k)$ is used for the first 1000 iterations (as a warm start) as follows:
  \begin{align*}
\alpha_k := \begin{cases}
0 & {\rm if }\ {G_{{\bm z}^k,\lambda_k}({\bm x})}=\|{\bm v}^k-{\bm x}^k\|=0,\\
\displaystyle \min\left\{1,\frac{\iota(k)\cdot G_{{\bm z}^k,\lambda_k}({\bm x}^k)}{(L_\Psi({\bm x}^k,{\bm z}^k)+L_f)\|{\bm v}^k-{\bm x}^k\|^2}\right\} & {\rm otherwise},
\end{cases}
\end{align*}
where $\iota(k)=1020$ if $k\leq 900$ and $\iota(k)=\max\{1020 - (k-900)\times 10.2,1\}$ if $k> 900$.\footnote{This means that $\iota(k)$ drops from 1020 for $k = 900$ to $1$ for $k = 1000$, and stays constant from then on.}
\end{itemize}
Both Algorithm~\ref{alg} and CoexDurCG start with the same initial point ${\bm x}^0=\frac{1}{n} {\bm 1}$ (the barycenter of the Birkhoff polytope) and the same initial multiplier ${\bm z}^0= [27,27]^T$; in addition, the LMO of $\mathcal{C}$ is computed using a subroutine of the solver \cite{FW.jl} based on the Hungarian algorithm \cite{Kuhn-1955}.  For the Gurobi solver, all parameters are under the default setting, such as the initial point and stopping criterion.

We compare our Algorithm~\ref{alg} with CoexDurCG via the relative primal function value measure and the relative feasibility measure,  respectively defined as
    \[
    {\rm val}_k:=\frac{|f({\bm x}^k)-f^*_{\rm es}|}{\max\{f^*_{\rm es},1\}},  \quad {\rm feas}_k:=\max\left\{\frac{\|[{\bm g}({\bm x}^k)]_+\|_{\infty}}{\max\{\|{\bm g}({\bm 0})\|_\infty,1\}},10^{-8}\right\},
    \]
    where $f^*_{\rm es}>0$ is the estimated optimal value returned by Gurobi.\footnote{Of all runs, Gurobi returned the lowest function value on an approximately-feasible point, and hence this value is used as a benchmark.}
Figure~\ref{fig:open-loop-xz} exhibits the performance of Algorithm~\ref{alg}(OL),  Algorithm~\ref{alg}(SS) and CoexDurCG in terms of ${\rm val}_k$ and ${\rm feas}_k$ over iterations.   The plots in Figure~\ref{fig:open-loop-xz} show the best-observed behavior for each respective algorithm variant. While Algorithm~\ref{alg}(OL) provided the best function value overall, Algorithm~\ref{alg}(SS) had a stable improvement of function value and the feasibility over iterations.  Overall, all algorithms were able to produce a reasonable feasible approximate solution, although CoexDurCG had the highest relative primal function value errors. 
Moreover,  from the table in Figure~\ref{fig:open-loop-xz}, both Algorithm~\ref{alg}(OL) and Algorithm~\ref{alg}(SS) are more than 4 times faster than the Gurobi solver, though Gurobi provided a slightly better function value. This shows that Algorithm~\ref{alg} is competitive when a high solution accuracy is not the primary concern.

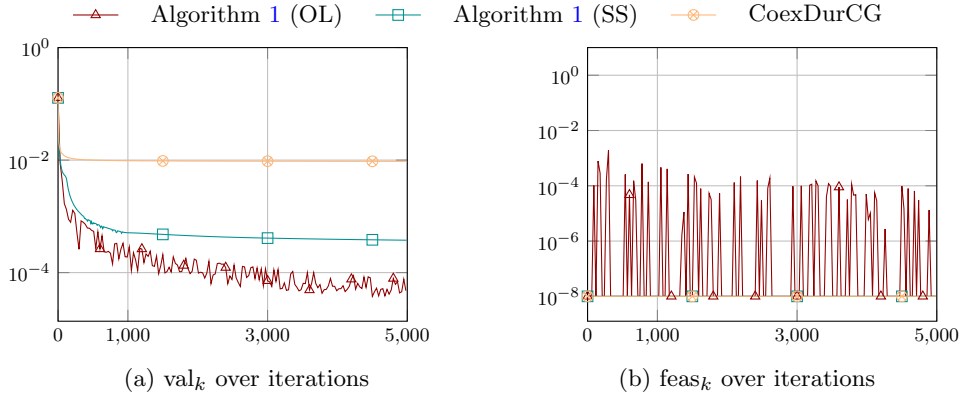
\begin{figure}[!h]
\begin{subfigure}{.5\linewidth}
    \centering
  \begin{tikzpicture}[scale=1]
\begin{axis}[height=5.2cm,width=6.2cm, legend cell align={left},
ylabel style={yshift=-0.75em,font=\scriptsize}, 
xlabel style={yshift=0.50em,font=\scriptsize},
tick label style={font=\scriptsize, scale=0.8, transform shape},
legend
entries={Algorithm~\ref{alg} (OL),Algorithm~\ref{alg} (SS), CoexDurCG},
legend style={
   font=\scriptsize,
   at={(2.4,1.2)},
   legend columns=-1,
   column sep= 1em,
   anchor=north east,
   draw = none,
   scale=0.5,         
   transform shape,
   fill = none,
   },
ytick={0,1e-8,1e-6,1e-4,1e-2,1},
xtick={0,1000,3000,5000},
grid,
xmin =0, xmax=5000, ymin=0, ymax=1, ymode=log,
mark repeat=50]
\addplot+[thin, each nth point=3, mark=triangle, mark repeat=20, color=cb-burgundy] table[x={iter}, y={rel_primal_gap}] {FWALo-n500-5000-05-09T17-28.txt};
\addplot+[thin, mark=square, mark repeat=150, color=cb-green-sea] table[x={iter}, y={rel_primal_gap}] {FWALs-n500-5000-05-12T22-44.txt};
 \addplot+[thin,solid, mark=otimes, mark repeat=150, color=cb-salmon-pink]
 table[x={iter}, y={rel_primal_gap}] {CD-n500-5000-05-06T13-24.txt};
\end{axis}
\end{tikzpicture}
\caption{${\rm val}_k$ over iterations}
\label{fig:sub1}
\end{subfigure}
\begin{subfigure}{.5\linewidth}
    \centering
  \begin{tikzpicture}[scale=1]
\begin{axis}[height=5.2cm,width=6.2cm, legend cell align={left},
ylabel style={yshift=-0.75em,font=\scriptsize}, 
xlabel style={yshift=0.50em,font=\scriptsize},
tick label style={font=\scriptsize, scale=0.8, transform shape},
ytick={0,1e-8,1e-6,1e-4,1e-2,1},
xtick={0,1000,3000,5000},
grid,
ymode=log,
xmin =0, xmax=5000, ymin=0, ymax=10, mark repeat=50]
\addplot+[thin, each nth point=3,  mark=triangle, mark repeat=20, color=cb-burgundy] table[x={iter}, y={rel_feas}] {FWALo-n500-5000-05-09T17-28.txt};
\addplot+[thin, mark=square, mark repeat=150, color=cb-green-sea] table[x={iter}, y={rel_feas}] {FWALs-n500-5000-05-12T22-44.txt};
 \addplot+[thin,solid, mark=otimes, mark repeat=150, color=cb-salmon-pink]
 table[x={iter}, y={rel_feas}] {CD-n500-5000-05-06T13-24.txt};
\end{axis}
\end{tikzpicture}
\caption{${\rm feas}_k$ over iterations}
\label{fig:sub1}
\end{subfigure}
\begin{center}
	\begin{tabular}{cccccc}
	\hline
	Algorithms  &  iterations & objective & feasibility & ${\rm val}_{5000}$ & time (sec) \\
	\hline
Algorithm~\ref{alg} (OL)	&  5000  & 123414.30 &  0  &  $4.75\times 10^{-5}$ &  1260.93 \\
Algorithm~\ref{alg} (SS)	&  5000 & 123454.93&  0 & $ 3.77\times 10^{-4}$  &   852.82\\
CoexDurCG	& 5000 & 124580.37 &  0  & $9.45\times 10^{-3}$ & 2835.03  \\
Gurobi	&      18     &  123408.44 &  $6.87\times 10^{-13}$ & 0 & 5108.93   \\
	\hline
\end{tabular}	
\end{center}
    \caption{We set $(n,m)=(500,2)$ in \eqref{e:num}, i.e, a quarter of a million variables, two quadratic constraints and the Birkhoff polytope constraint.}
    \label{fig:open-loop-xz}
\end{figure}

\noindent {\bf Acknowledgements.} The work of the first author was partially supported  by the National Natural Science Foundation of China (12501427). The work of the second author was supported in part by the Hong Kong Research Grants Council PolyU 15300423. The work of the third author was supported by the National Science Foundation under grant DMS-253242.

\noindent {\bf Data availability.} 
The code used to generate our numerical results can be found at the following link: \url{https://github.com/zevwoodstock/CGALI}.

\noindent {\bf Competing interests.} The authors declare no competing interests.

\bibliography{references_alm_nonlinear}

\end{document}